\documentclass{amsart}
\usepackage{amsmath,amsthm,amsfonts,amscd,amssymb,eucal,latexsym,mathrsfs, tikz, pgf,  mathabx, epsfig, standalone, graphics, graphicx, pstricks, subcaption, caption, mathtools}

\usepackage{subfiles}

\usetikzlibrary{arrows}
\usetikzlibrary{matrix}

\theoremstyle{plain}
\newtheorem{theorem}{Theorem}
\newtheorem{corollary}[theorem]{Corollary}
\newtheorem{lemma}[theorem]{Lemma}
\newtheorem{proposition}[theorem]{Proposition}
\theoremstyle{definition}
\newtheorem{notation}[theorem]{Notation}
\newtheorem{remark}[theorem]{Remark}

\newtheorem{definition}[theorem]{Definition}

\newtheorem{claim}[theorem]{Claim}
\allowdisplaybreaks[2]

\numberwithin{theorem}{section}
\numberwithin{figure}{section}

\title{The (2,4,5) Triangle Coxeter Group is not Systolic}
\author{ADAM WILKS}

\begin{document}
\maketitle

\begin{abstract}
  We show that the (2,4,5) triangle Coxeter group is not systolic.
\end{abstract}
\tableofcontents 

\section{Introduction}

We begin with the definition of a systolic complex, introduced and studied independently by Haglund \cite{Hag03} and Januszkiewicz--\'{S}wi\c{a}tkowski \cite{JS06}. The 1--skeleta of systolic complexes were studied even earlier by Chepoi and other under the name of \textit{bridged graphs} (see \cite{BC08}). A simplicial complex $X$ is \textit{flag} if any clique in $X$ spans a simplex of $X$. A subcomplex $K$ of $X$ is \textit{full} if any simplex of $X$ spanned by vertices in $K$, also is contained in $K$. A flag simplicial complex $X$ is \textit{k-large}, if there are no embedded cycles of length less than $k$ and greater than $3$ which are full subcomplexes of $X$. A simplicial complex $X$ is \textit{systolic} if it is connected, simply connected, flag, and 6-large. A group is \textit{systolic} if it acts geometrically, ie properly and cocompactly, by simplicial automorphisms on some systolic complex $X$.

Our paper  produces one of the few examples of a hyperbolic group that is not systolic. Previously, the only known obstruction for hyperbolic groups to be systolic was given by Januszkiewicz and \'{S}wi\c{a}tkowski (see \cite{JS07}). In their paper they introduced \textit{filling invariants} that excluded some groups from being systolic. We cannot use their techniques, instead the main tools we use are the systolic fixed point theorem, Chepoi--Osajda \cite{CO12} , and the systolic Projection Lemma, first shown by Januszkiewicz and \'{S}wi\c{a}tkowski in \cite{JS06}. 

Our example comes from the class of triangle Coxeter groups, which are groups of the form
\[
	W = \langle r,s ,t | r^2=s^2=t^2=(rs)^k=(st)^l=(tr)^m = 1 \rangle
\] where the exponents $k,l,m \geq 2$. Previously Przytycki and Schwer in \cite{PS13} showed that all triangle Coxeter groups were systolic, except for those with exponents $(2,4,4), (2,4,5), $ and $(2,5,5)$. The group with exponents $(2,4,4)$ was shown to not be systolic using the flat torus theorem following the proof of Theorem 4.1 in \cite{EP13}, but whether the groups with exponents $(2,4,5)$ and $(2,5,5)$ are systolic was left open. Our paper deals, in particular, with the triangle Coxeter group with exponents $(2,4,5)$, but we believe our techniques can be adapted to handle the remaining group as well. For the rest of the paper we refer to the triangle Coxeter group with exponents $(2,4,5)$ by $W$ and we will fix the presentation as
\[
	W = \langle r, s , t | r^2 =s^2=t^2=(rs)^4=(st)^5=(tr)^2 \rangle.
\] Our main result is the following theorem.
\begin{theorem}
	W is not systolic.
	\label{thm:main}
\end{theorem}
 A crucial intermediary step that is needed in the proof of Theorem~\ref{thm:main} is a strengthened version of the fixed point theorem for $W$.
\begin{theorem}
\label{Thm:fix}
	If $W$ acts  geometrically by simplicial automorphisms on a systolic complex $X$, then there exist simplices $A,B,C$ of $X$ such that $A$ is fixed by $\langle r, s \rangle$, $B$ is fixed by $\langle s, t \rangle$, $C$ is fixed by $\langle t, r \rangle$, and $A \cap B \cap C \neq \emptyset$. 
\end{theorem}

\subsection{Organisation}

Preliminary definitions, notation, and lemmas are given in Section~\ref{sec:prelim}. The outline of the proof of Theorem~\ref{Thm:fix} is given in Section~\ref{sec:fix} with the details being filled in by Sections~\ref{sec:base}~and~\ref{sec:ind}. The proof of Theorem~\ref{thm:main} is given in Section~\ref{sec:proofMain}, postponing the proofs of some of the intermediary results to Sections~\ref{sec:connectingCliques}--\ref{sec:(5,4)}.

\section{Systolic Spaces and Edge Types}

\label{sec:prelim}
\begin{notation}
	From now on all actions are implied to be by simplicial automorphisms.
\end{notation}

\begin{definition}
   \label{Def1}
	Suppose $G$ is a group, $X$ is a systolic space, and that $G$ acts on $X$. An $\textit{edge}$ in $X$ is a pair of 0--cells (not necessarily adjacent). The $\textit{edge type}$ of the edge $(u,v)$ is the $G-$orbit $G(u,v)$. We say that the edge $(u,v)$ is $\textit{contained}$ in $X$ if $u$ and $v$ are adjacent in $X$. We say that the edge type of $(u,v)$ is \textit{contained} in $X$ if every edge in the edge type of $(u,v)$ is contained in $X$.
\end{definition}

    In Defintion \ref{Def1} since the action of $G$ is simplicial if some edge $(u,v)$ is contained in $X$ then all the edges in the edge type of $(u,v)$ are contained in $X$. For this reason when we say $(u,v)$ is contained in $X$, the reader should take this to mean that both the edge $(u,v)$ is contained in $X$ and that the edge type of $(u,v)$ is contained in $X$.

    For convenience we say that $(v,v)$ is contained in $X$ for all $v \in X^0$.

\begin{definition}
	An \textit{n-cycle} is a sequence of vertices in $X$ of length $n$. We say that $X$ \textit{contains} the $n$-cycle, $v_1, \dots , v_n$ if $X$ contains the edge types $(v_i, v_{i+1})$ for $i < n$ and the edge type $(v_n,v_1)$. For an $n$-cycle $\gamma$ we say that a \textit{side} of $\gamma$ is an edge between consecutive vertices in $\gamma$ and a \textit{diagonal} of $\gamma$ is an edge between nonconsecutive vertices in $\gamma$. An \textit{n-cycle type} is the $G$--orbit of an $n$-cycle.
\end{definition}

\begin{remark}
	\label{rm:sys -> diag}
	Suppose $X$ is a systolic complex. Since $X$ is a $6$-large complex, if $X$ contains a 4-cycle  or a 5-cycle $C$, then $X$ contains one of the diagonals of $C$. This remarks is still true if the cycle is not embedded in $X$. For instance if two consecutive vertices in a $4$ or $5$ cycle $\gamma$ are equal then some of the sides of $\gamma$ are also  diagonals, so the remark is trivially true.  
\end{remark}

\begin{definition}
    Let $S$ be a triangulated surface with boundary. For $v \in S^0$ let $\angle(v)$ denote the number of 2-cells containing $v$. If $v$ is in the interior of $S$ we say $v$ has \textit{curvature} $\kappa(v) = 6 - \angle(v)$. If $v$ is on the boundary of $S$ then we say that $v$ has \textit{boundary curvature} $\kappa_{\partial}(v) = 3 - \angle(v)$.
\end{definition}

\begin{lemma} {(Gauss--Bonnet Theorem)}
\label{lem:gauss}
	If $S$ is a compact triangulated surface, then

	\[
        6 \chi (S)  = \sum_{v \in \partial S} \kappa_{\partial}(v) + \sum_{v \in IntS} \kappa(v)
    \]
\end{lemma}

As mentioned by the authors in the introduction for \cite{EP13} a direct consequence of Theorem C of their paper is the systolic fixed point theorem, stated now.

\begin{theorem} [systolic fixed point theorem]
\label{thm:fix}
	Let $G$ be a finite group acting on a systolic complex $X$. Then there exists a simplex $\sigma \in X$ which is invariant under the action of $G$.
\end{theorem}

We follow Definition 2.6 and Theorem 2.7 from \cite{P14} in giving a formulation of a filling diagram and their existence in systolic complexes.

\begin{definition}
	A simplicial complex $X$ is \textit{locally k-large} if the link of every vertex in $X$ is $k$-large. 
\end{definition}

\begin{lemma}
\label{lem:glob->loc}
	Suppose a simplicial complex $X$ is $k$-large. Then $X$ is locally $k$-large. In particular a systolic complex is locally $6$-large.
\end{lemma}

\begin{definition}
	Let $\gamma$ be a a cycle in a simplicial complex $X$. A \textit{filling diagram} for $\gamma$ is a simplicial map $f: D \rightarrow X$, where $D$ is a triangulated 2-disc, and $f|_{\partial D}$ maps isomorphically onto $\gamma$. Fix a filling diagram $f$. We say $f$ is \textit{minimal} if $D$ consists of the least possible number of 2-simplices among filling diagrams for $\gamma$. We say $f$ is \textit{locally k-large} if $D$ is locally $k$-large as simplicial complex. We say $f$ is \textit{nondegenerate} if $f$ is injective on each simplex.
\end{definition}

\begin{theorem}
\label{thm:exist}
	Let $\gamma$ be a cycle in a systolic complex $X$. Then we have the following:
	\begin{enumerate}
		\item there exists a filling diagram for $\gamma$,
		\item any minimal filling diagram for $\gamma$ is simplicial, locally $6$-large and nondegenerate 
	\end{enumerate}
\end{theorem}

\begin{remark}
\label{rem:curvInt}
	Suppose $f:D \rightarrow X$ is a minimal filling diagram for a cycle in a systolic complex $X$. Since $D$ is locally $6$-large for any $v$ in the interior of $D$ we have that $\angle(v) \geq 6$ and consequently $\kappa(v) \leq 0$. 
\end{remark}

We state a projection lemma for systolic space given by
Lemma 7.9 (1) in \cite{JS06}.

\begin{definition}
	Suppose $X$ is a simplicial complex, $\sigma$ is a simplex in $X$, and $Y$ is a subcomplex of $X$. The \textit{residue} of $\sigma$ in $X$, Res$(\sigma,X)$ is the union of all the simplices of $X$ that contain $\sigma$. 
\end{definition}
	
\begin{definition}
	A subcomplex $Q$ in a simplicial complex $X$ is \textit{convex} if it is full and given any geodesic path $\gamma$ in the $1$-skeleton of $X$, with both endpoints in $Q$, the path $\gamma$ is contained in $Q$. Obviously a $1$--cell in $X$ is convex.
\end{definition}

\begin{lemma}[Projection Lemma]
\label{lem:proj}
	Let $X$ be a systolic simplicial complex, $Q$ a convex subcomplex in $X$, and $n \geq 1$ a natural number. Then for any simplex $\rho$ of the sphere $S_n(Q,X)$ the intersection $B_{n-1}(Q,x) \cap \text{Res}(\rho,X)$ is a single, nonempty simplex of $X$.
\end{lemma}

% \begin{lemma}
%         Let $G = \langles,t,r | (st)^5= (tr)^2 = (rs)^4 = 1>$ act geometricaly on some systolic complex $X$.There exist cliques $A$ and $V$ such that they are fixed by $\langles,t>$ and $\langler,s>$ respectively and $A \cap B \neq \emptyset$.
% \end{lemma}

% \begin{proof}
%         For $x,y \in X^0$ let $d(x,y)$ be the distance between $x$ and $y$ in the 1-skeleton of $X$. Let $v$ and $w$ be 0-cells in $A$ and $B$ respectively minimizing $d(v,w)$. We proceed by induction on $d(v,w)$. 

%         {\emph Base Case: d(v,w) = 1}   

%  \end{proof}

\section{Proving Theorem \ref{thm:fix}}

\label{sec:fix}

In this section we give an outline for proving Theorem~\ref{Thm:fix}. The main difficulties of the proof lie in the proof of the following propositions, which we prove in subsequent sections.

\begin{proposition}
\label{prop:base}
    Suppose that $W$ acts geometrically on a systolic complex $X$. If there exist cliques $A$ and $B$ fixed by $\langle s,t\rangle$ and $\langle r,s\rangle$, respectively, such that $d(A,B) =1$, then there exist cliques $A'$ and $B'$ fixed by $\langle s,t \rangle$ and $\langle r,s \rangle$, respectively,  such that $A' \cap B' \neq \emptyset$.
 \end{proposition}

 \begin{proposition}
 \label{prop:ind}
     Suppose that $W$ acts geometrically on a systolic complex $X$. If there exist cliques $A$ and $B$ fixed by $\langle s,t\rangle$ and $\langle r,s\rangle$, respectively, such that $d(A,B) =n$, for $n \geq 2$, then there exists cliques $A'$ and $B'$ fixed by $\langle s,t \rangle$ and $\langle r,s \rangle$, respectively,  such that $d(A',B') < n$.
 \end{proposition}

 With these two propositions we can prove Theorem \ref{Thm:fix} in the following way:

 \begin{proof} [Proof of Theorem \ref{Thm:fix}]
     By assumption we have that $W$ acts on a systolic complex $X$. By Theorem~\ref{thm:fix} there are cliques $A$ and $B$ fixed by the finite subgroups $\langle s,t\rangle$ and $\langle r,s\rangle$, respectively. We prove that there exists intersecting $A'$ and $B'$ fixed by $\langle s,t\rangle$ and $\langle r,s\rangle$ respectively by induction on $d(A,B)= n$. The base case for $n=1$ is Proposition~\ref{prop:base} and the induction step is Proposition~\ref{prop:ind}.

     Now let $v \in A' \cap B'$ and let $C = \langle r, t \rangle v$. By construction $C$ is invariant under $\langle r, t \rangle$ and intersects with $A'\cap B'$, so for $C$ to satisfy conditions of Theorem~\ref{Thm:fix} all that remains is to show that $C$ is a clique. Since $A'$ and $B'$ are cliques, we see that $v$ is adjacent or equal to $rv$ and $tv$. Therefore $X$ contains $(v,rv)$ and $(v,tv)$, so the vertices of $C$ form a cycle $\gamma$ contained in $X$. If $\gamma$ is of length $1$, $2$, or $3$ then it has no diagonals so $C$ is a clique. If $\gamma$ is of length $4$ and then by Remark~\ref{rm:sys -> diag} $X$ contains one of the diagonals of $\gamma$. Since both of the diagonals of $\gamma$ are in the same edge type, we get that both of the diagonals of $\gamma$ are contained in $X$. Therefore $C$ is a clique.     
 \end{proof}

\section{Base Case}
\label{sec:base}

In this section we prove Proposition~\ref{prop:base}. In the first part of the section we prove a number of structural lemmas used in the proof, which constitutes the second half of the section.

\begin{lemma}
    \label{D8L0}
    Suppose $\gamma$ is a $6$-cycle embedded in a systolic complex $X$. Then either $X$ contains a diagonal of $\gamma$ or there exists $v \in X^0$ adjacent to all the vertices of $\gamma$.
\end{lemma}

\begin{proof}
    Suppose $X$ does not contain a diagonal of $\gamma$. By Theorem~$\ref{thm:exist}$, there exists a filling diagram of $\gamma$. By Remark~\ref{rem:curvInt} every vertex in the interior of $D$ has nonpositive curvature. Therefore by Lemma~\ref{lem:gauss}
    \[
        \sum_{v \in \partial D} \kappa_{\partial}(v) \geq 6 .
    \] Since $X$ does not contain a diagonal of $\gamma$ for every $v \in \partial D$ $\kappa_{\partial}(v) \leq 1$. Combining this fact with the inequality on the sum of the boundary curvatures we see that all the boundary vertices have boundary curvature exactly $1$. Therefore they are all adjacent to a common vertex in $D$.
\end{proof}

\begin{lemma}

\label{D8L1}
    Suppose $\gamma$ is an $8$-cycle in a systolic complex $X$ and $f: D \rightarrow X$ is a minimal filling diagram of $\gamma$. Then either $D$ is isomorphic to the diagram $P_8$ on the left in Figure~\ref{octs} or is isomorophic to the diagram $P_{10}$ on the right in Figure~\ref{octs} or some 0-cell on the boundary of $D$ has curvature $2$.
    \label{lem:P8_or_P10}
\end{lemma}

\begin{figure} [!ht]
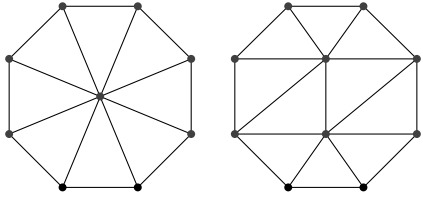

  \include{Graphics/Oct1}
  \caption{$P_8$ and $P_{10}$}
  \label{octs}
\end{figure}

\begin{proof}
    Let the vertices of $\gamma$ be labeled $v_0, \dots , v_7$. By Remark~\ref{rem:curvInt} every vertex in the interior of $D$ has nonpositive curvature. Therefore by Lemma~\ref{lem:gauss} we have
    \[
        \sum_{v \in \partial D} \kappa_{\partial}(v) \geq 6 .
    \]

     Suppose no 0-cell on the boundary of $D$ has curvature  2. Then at least 6 of the 0-cells on the boundary of $D$ have curvature $1$. Let $a$ and $b$ be the remaining two 0-cells on the boundary of $D$, which do not necessarily have to have curvature 1. We can assume up to relabeling that $a= v_0$. 

     First suppose that $b = v_1$. We see that all the 0-cells on the boundary of $D$ are adjacent to a common vertex. Therefore $D = P_8$.

     Now suppose $b = v_2$. Then $v_2, \dots ,v_7, v_0$  are adjacent to a common 0-cell, say $w$. Their configuration is shown in Figure \ref{oct2} . We see that $D$ contains the 4-cycle $v_1, v_2, w,v_0$, so $D$ must contain either $(v_2, v_0)$ or $(v_1,w)$. If $v_2$ is adjacent to $v_0$ then $v_1$ would have curvature 2 contradicting our assumption so $v_1$ is adjacent to $w$. Therefore $D = P_8$.

     \begin{figure} [!ht]
     \include{Graphics/Oct2}
     \caption{}
    \label{oct2}
\end{figure}

     Now suppose $b = v_3$. Then $v_3,v_4,v_5,v_6,v_7,v_0$ are adjacent to some common vertex, say $w$, and $v_0, v_1, v_2,v_3$ are connected to some common  vertex, say $x$. If $w=x$ then we are done otherwise their configuration is given by Figure  \ref{oct3} . We see that $D$ contains the 4-cycle $v_0, x, v_3, w$, therefore $D$ contains $(v_0, v_3)$ or $(x,w)$. In either case $\kappa(x) > 0$, which gives a contradiction. 

     \begin{figure}[!ht]
     \include{Graphics/Oct3}
     \caption{}
     \label{oct3}
     \end{figure}

     Now suppose $b= v_4$. Then $v_4,v_5,v_6,v_7,v_0$ are connected to a common  vertex, say $w$, and $v_0,v_1,v_2,v_3,v_4$ are connected to a common  vertex, say $x$. If $w =x$ then we are done, otherwise the configuration of $D$ is given by Figure~\ref{oct4}. We see that $D$ contains the $4$-cycle $v_0,x,v_4,w$, so $D$ contains either $(v_0,v_4)$ or $(x,w)$. If $D$ contains $(v_0,v_4)$ then $\kappa(x) > 0$, which gives a contradiction. Therefore $D$ contains $(x,w)$, so $D = P_{10}$.

     \begin{figure} [!ht]
     \include{Graphics/Oct4}
     \caption{}
     \label{oct4}
     \end{figure}

 \end{proof}

 \begin{lemma}

\label{lem:D8_diag->clique}
    Suppose $X$ is a systolic complex. Let $D_8$ be the dihedral group with presentation $\langle a,b| (ab)^4 = a^2 = b^2 = 1\rangle$ and assume it acts on $X$. Let $v \in X$ and assume that the cycle $v, av, abv, abav , \dots , abababav$ is contained in $X$, call it $O$. If any two non-consecutive vertices in the cycle $O$ are adjacent, then $O$ is a clique.  
\end{lemma}

\begin{proof}
      We use the labels $2,3,3',$ and $4$ to respectively refer to the edge types $(v, abv)$, $(v, abav)$, $(v, babv)$, and $(v, ababv)$. The labeling is shown in Figure~\ref{Dihedral1}. We use the notation $n \Rightarrow m$ to say that if $X$ contains the edge type $n$ then it needs to contain the edge type $m$. We claim the following implications.

    \[
        3 \text{ or } 3' \Rightarrow 2 \Rightarrow 4 \Rightarrow 3 \text{ and } 3' 
    \]
    These implications finish the proof because given that $X$ contains any edge type all the other edge types are implied to also be contained by $X$. Now we prove each implication individually.

    \begin{figure}
\begin{minipage}{0.35\textwidth}
                  \include{Graphics/Dihedral1}
                \caption{}
                \label{Dihedral1}
            \end{minipage}%
\begin{minipage}{0.35\textwidth}
               \include{Graphics/Dihedral2}
                \caption{}
                \label{Dihedral2}
            \end{minipage}%
\begin{minipage}{0.35\textwidth}
                \include{Graphics/Dihedral3}
                \caption{}
                \label{Dihedral3}
            \end{minipage}
\end{figure}

            Suppose that $X$ contains $2$. Then $X$ contains the 4-cycle $v,abv,ababv,bav$ shown in Figure~\ref{Dihedral2}. The 4-cycle has diagonals in $4$, so $X$ contains $4$.

             Suppose $X$ contains $4$. Then $X$ contains the 4-cycle $v,av, babv, ababv$, shown in  Figure \ref{Dihedral3}. The 4-cycle has diagonals in $3'$, so $X$ contains $3'$. Also, we see that $X$ contains the 4-cycle $v,bv,abav,ababv$, which has diagonals in $3$, so $X$ contains $3$.

            Suppose $X$ contains $3$. Then $X$ contains the 4-cycle $v,av,abv,abav$, which has diagonal in $2$ so $X$ contains $2$. Now suppose that $X$ contains $3'$, then $X$ contains the 4-cycle $v,bv, bav, babv$, which has diagonals in $2$, so $X$ contains $2$.
\end{proof}

\begin{lemma}
    Suppose $X$ is a systolic complex. Let $D_{10}$ be the dihedral group with presentation $\langle a,b| (ab)^5 = a^2 = b^2 = 1\rangle$ and assume it acts on $X$. Let $v \in X^0$ and assume that the 10-cycle $v, av, abv, abav , \dots , (ab)^4av$ is contained in $X$, call it $O$. If any two non-consecutive vertices in the cycle $O$ are adjacent then $O$ is a clique. 

    \label{lem:10cycle-diag->clique}
\end{lemma}

\begin{proof}
We label the edge types $(v,abv), (v,abav), (v,babv), (v,ababv), (v,ababav)$ as $2,3,3',4,$ and $5$. This labeling is shown in Figure \ref{10Dihedral1}.

\begin{figure} [!ht]
                \include{Graphics/10Dihedral1}
                \caption{}
                \label{10Dihedral1}
            \end{figure}

Using the same notation as we did in the proof of the previous lemma we claim the following implications hold,
\[
    3 \text{ or } 3' \Rightarrow 2 \Rightarrow 4 \Rightarrow 5 \Rightarrow 3 \text{ and } 3'
\]

Suppose that $X$ contains $3$. Then $X$ contains the 4-cycle $v, abav, abv, av$ shown in Figure \ref{10Dihedral2}, therefore contains one of its diagonals both of which are in $2$. Now suppose $X$ contains $3'$. Then $X$ contains the 4-cycle $v,babv,bav,bv$, therefore it contains one of its diagonal, both of which are in $2$.

\begin{figure}
\begin{minipage}{0.42\textwidth}
                  \include{Graphics/10Dihedral2}
               \caption{}
                \label{10Dihedral2}
            \end{minipage}%
\begin{minipage}{0.42\textwidth}
                \include{Graphics/10Dihedral3}
                  \caption{}
                \label{10Dihedral3}
            \end{minipage}%
\begin{minipage}{0.42\textwidth}
                \include{Graphics/10Dihedral4}
                \caption{}
                \label{10Dihedral4}
            \end{minipage}
\end{figure}

Suppose that $X$ contains $2$. Then $X$ contains the 5-cycle $v,abv, ababv,babav,bav$ shown in Figure \ref{10Dihedral3}. All the cycle's diagonals are in $4$, so $X$ contains $4$.

Suppose $X$ contains $4$. Then $X$ contains the 4-cycle $v,ababv, ababav,bv$ shown in Figure \ref{10Dihedral4}. $X$ has to contain at least one of the 4-cycle's diagonals and both diagonals are in $5$, so $X$ contains $5$.

Suppose $X$ contains $5$. Then $X$ contains the 4-cycle $v,av,babav,ababav$ shown in Figure \ref{10Dihedral5}. This figure has diagonals in $4$ so $X$ contains $4$. Therefore $X$ contains the 5-cycle $v, av, abv, abav, ababv$ shown in Figure \ref{10Dihedral6}. Any triangulation of this 5-cycle contains an edge in $2$, so $X$ contains $2$. Therefore $X$ contains the 4-cycle $v,abv,babv,ababav$, shown in Figure \ref{10Dihedral7}. The diagonals of the 4-cycle are in $3'$ therefore $X$ contains $3'$. Also $X$ contains the 4-cycle $v,abv,abav,ababav$, which has diagonals in $3$ therefore $X$ contains $3$. 

\begin{figure}
\begin{minipage}{0.42\textwidth}
                \include{Graphics/10Dihedral5}
                  \caption{}
                \label{10Dihedral5}
            \end{minipage}%
\begin{minipage}{0.42\textwidth}
                \include{Graphics/10Dihedral6}
                  \caption{}
                \label{10Dihedral6}
            \end{minipage}%
\begin{minipage}{0.42\textwidth}
                \include{Graphics/10Dihedral7}
                  \caption{}
                \label{10Dihedral7}
            \end{minipage}
\end{figure}

\end{proof}

\begin{lemma}
Suppose $X$ is a systolic complex and $O$ is an 8-cycle embedded in $X$. Suppose $r$ and $t$ are simplicial automorphisms of $X$, which reflect $O$ over the axes shown Figure~\ref{BCase5G}. Then there cannot be a minimal filling diagram of $O$ with domain $P_{10}$.
% Suppose there exists a filling diagram of $O$, with domain $P_{10}$. Then either $X$ contains $(v,tv)$ or $(w,rw)$ or there exists a filling diagram of $O$, with domain $P_8$.
\label{lem:8cyl:P8->P10}  
\end{lemma}

\begin{figure}[ht!]
    \include{Graphics/BCase5G}
    \caption{}
    \label{BCase5G}
\end{figure}

\begin{figure} [!tbp]
    \centering
    \begin{minipage}{0.3\textwidth}
    \centering
        \include{Graphics/Case5B}
        \caption{}
        \label{Case5B}     
    \end{minipage}%
       \begin{minipage}{0.3\textwidth}
    \centering
        \include{Graphics/BCase5C}
         \caption{}
        \label{Case5C}
    \end{minipage}%
    \begin{minipage}{0.3\textwidth}
    \centering
        \include{Graphics/BCase5A}
        \caption{}
        \label{Case5A}
    \end{minipage}%
    \begin{minipage}{0.3\textwidth}
    \centering
        \include{Graphics/BCase5D}
         \caption{}
        \label{Case5D}
    \end{minipage}
\end{figure} 

\begin{proof}
Suppose there exists a minimal filling diagram $f$ with domain $P_{10}$. We label the internal vertices of $P_{10}$  $x$ and $y$ as shown in Figure \ref{Case5B}. We see that $X$ contains the 4-cycle $ry,tv,x,v$ shown in Figure \ref{Case5C}. 

Suppose $X$ contains $(v,tv)$ one of the diagonals of the $4$-cycle $ry, tv,x,v$. Then we can construct a filling diagram of $O$ by constructing a filling diagram of the $6$-cycle $v,tv,trv,rw,rw,rv$ and a filling diagram of the $4$-cycle $v,w,tw,tv$. The $4$-cycle has a diagonal, so has a filling diagram with a domain containing two $2$-cells. By Lemma~\ref{D8L0} either $X$ contains a diagonal of the $6$-cycle or there is a vertex adjacent to all the vertices in the $6$-cycle. Either way we can produce a filling diagram of the $6$-cycle with a domain containing at most six $2$-cells. Therefore we can produce a filling diagram of $O$ with a domain containing at most eight $2$-cells contradicting the minimality of $f$.

Suppose $X$ contains $(x,ry)$, the other diagonal of the $4$-cycle $ry, tv, x,v$. By a similar argument, we can show $X$ contains $(y,rx)$. Therefore $X$ contains the 4-cycle $x,y,rx,ry$, so $X$ contains either $(x,rx)$ or $(y,ry)$. Without loss of generality assume $X$ contains $(y,ry)$ if instead $X$ contains $(x,rx)$ we can proceed similarly (swapping $x$ with $y$ in the argument). Now we see that $X$ contains the 4-cycle $y,rv,v,ry$ seen in Figure \ref{Case5A}. Both of the cycle's diagonals are in the edge type $(y,v)$, so $X$ contains $(y,v)$. Therefore $X$ contains the 5-cycle $y,v,w,tw,tv$ seen in Figure \ref{Case5D}. Any $5$-cycle in a systolic complex has a filling diagram with a domain containing no more than three $2$-cells. Therefore we can produce a filling diagram of $O$ with a domain containing no more than eight $2$-cells contradicting the minimality of $f$.
\end{proof}

\begin{lemma}

Suppose $X$ is a systolic complex, which $W$ acts geometrically on. Let $v \in X^0$ and assume the 8-cycle $O$ given by $v, rv,rsv, \dots, r(sr)^3v$ is contained in $X$. Then either $O$ is a clique or  there is a filling diagram $f:P_8 \rightarrow X$ of $O$.
\label{lem:8cycl_P8_or_clique} 

\end{lemma}

\begin{proof}
    If $X$ contains a diagonal of $O$ then by Lemma~\ref{lem:D8_diag->clique} $O$ is a clique, so assume not. Therefore by Lemma~\ref{lem:P8_or_P10} we know there exists a minimal filling diagram of $O$, $f:D \rightarrow X$, such that $D= P_8$ or $P_{10}$. However if $D = P_{10}$ then we can apply Lemma~\ref{lem:8cyl:P8->P10} with $w = sv$ and $t = srs$ to get a contradiction. Therefore $D = P_8$.
\end{proof}

 \begin{lemma}
     Suppose $X$ is a systolic complex, which $W$ acts on geometrically. Let $v \in X^0$. Assume $\langle s,t\rangle v$ and $\langle t,r\rangle v$ are cliques. Then either $\langle s,r\rangle v$ is a clique or there is some other vertex $v'$ such that $\langle s,r\rangle v', \langle t,r\rangle v',$ and $\langle s,t\rangle v'$ are cliques.

     \label{lem:tv->cliques}
 \end{lemma}

 \begin{proof} 

   If $X$ contains any diagonal of $\langle s, r \rangle v$ then by Lemma~\ref{lem:D8_diag->clique} $\langle s, r \rangle v$ is a clique. Otherwise by Lemma~\ref{lem:8cycl_P8_or_clique} there exists a filling diagram of $\langle s, r \rangle v$, $f:P_8 \rightarrow X$. Let $w$ be the internal vertex in $P_8$. We see that $X$ contains the 4-cycles $w,rsrv, rstv,rv$ and $w,sv,tv,rv$ seen in Figure \ref{Th6P3}. These 4-cycles have  $(sv,rv)$ and  $(rsrv, rv)$ as diagonals. If either of these edges are contained in $X$  then by Lemma \ref{lem:D8_diag->clique} $\langle s,r\rangle v$ would be a clique. Therefore assume that the other two diagonals of these 4-cycles are contained in $X$ namely $(w,tv)$ and $(w,rstv)$. Then $X$ contains the 5-cycle $w,tv,tw,trsv,rstv$ shown in Figure~\ref{Th6P9}. If $X$ contains $(tv,trsv)$ then by Lemma \ref{lem:D8_diag->clique} $\langle s, r \rangle v$ is a clique, so assume not. Therefore $X$ needs to contain either the diagonal $(tv,rstv)$ or $(w,tw)$. 

   \begin{figure}
\begin{minipage}{0.4\textwidth}
\centering
                \include{Graphics/Th6P3}
                  \caption{}
                \label{Th6P3}
    \end{minipage}%
    \begin{minipage}{0.4\textwidth}
    \centering
       \include{Graphics/Thm6P2}
       \caption{}
       \label{Th6P2}
    \end{minipage}
    \end{figure}

    \begin{figure}
        \centering
        \include{Graphics/Th6P9}
                  \caption{}
                \label{Th6P9}
    \end{figure}

    First assume that $X$ contains $(tv,rstv)$. We claim that $v' = tv$ satisfies conditions of lemma. The edge $(tv,tsrv)$ is a diagonal of the 8-cycle $ \langle s, r \rangle v'$ as seen in Figure \ref{Th6P2}. Therefore by Lemma~\ref{lem:D8_diag->clique}  $ \langle s, r \rangle v'$ is a clique. By assumption $\langle s,t \rangle v'$ and $\langle r,t \rangle v'$ are cliques so we are done.

    Now assume that $X$ does not contain $(tv, rstv)$ so that $X$ contains $(w,tw)$. Then we claim that $v' = w$ satisfies the conditions of the lemma. If $w_1$ and $w_2$ are two 0-cells in $\langle r, s \rangle w$ then $X$ contains the 4-cycle $w_1,sv,w_2,rv$. Therefore $X$ contains $(w_1, w_2)$, so  $\langle r, s \rangle w$  is a clique. The same argument that we used to show that $X$ contains $(w,tw)$ can be used to show that $X$ contains $(sw,tw)$, which is a diagonal of the 10-cycle $\langle s, t \rangle w$ so by Lemma~\ref{lem:10cycle-diag->clique} the 10-cycle is a clique. The 4-cycle $\langle r, t \rangle w$ is trivially a clique. Therefore $w=v'$ satisfies the conditions of the lemma.
 \end{proof}

 \begin{lemma}
      Suppose $X$ is a systolic complex, which $W$ acts on geometrically. Let $v \in X^0$. Suppose the 10-cycle $O$ given by $v, sv, stv, \dots, (st)^4sv$, is contained in $X$. Then there is some vertex $w$ adjacent to every vertex of $O$ or $O$ is a clique.
      \label{lem:10-cyc-1Int}
 \end{lemma}

 \begin{proof}
     If $X$ contains any diagonal of $\langle s, t \rangle v$,  then $\langle s, t \rangle v$ is a clique by Lemma \ref{lem:10cycle-diag->clique} so assume not. We proceed by showing the existence of points adjacent to increasing collections of consecutive points of $\langle s, t \rangle v$.

     \begin{figure}
\begin{minipage}{0.4\textwidth}
\centering
                \include{Graphics/Lem10P1}
                  \caption{}
                \label{Lem10P1}
    \end{minipage}%
    \begin{minipage}{0.4\textwidth}
    \centering
       \include{Graphics/Lem10P2}
       \caption{}
       \label{Lem10P2}
    \end{minipage}
 \end{figure}

     First we show that there exists a $w \in X^0$ adjacent to $5$ consecutive points of the 10-cycle $\langle s, t \rangle v$. By Theorem~\ref{thm:exist} there exists a minimal filling diagram of $\langle s, t \rangle v$, $f:D \rightarrow X$. By Remark~\ref{rem:curvInt} for any $v \in $Int$D$, we have $\kappa(v) \leq 0$. This inequality together with Lemma~\ref{lem:gauss} gives us that
     \[
          \sum_{v \in \partial D} \kappa_{\partial}(v) \geq 6.
      \] If $3$ consecutive vertices on the boundary have curvature $1$, then there is a vertex adjacent to 5 consecutive 0-cells of the $\langle s, t \rangle v$, so assume not. No boundary vertex has curvature $2$ because by assumption $X$ contains no diagonals of $\langle s, t \rangle v$. If any boundary 0-cell has curvature less than 0, then the remaining $9$ boundary 0-cells would have curvature at least $7$, however there is no valid configuration of curvatures with the assumption that no 3 consecutive boundary vertices have curvature 1. Therefore all boundary vertices have curvature $0$ or $1$. Up to symmetry there are $3$ configurations of possible curvatures of the boundary vertices these are shown in Figure~\ref{Lem10P1}, Figure~\ref{Lem10P2}, and Figure~\ref{Lem10P3}.

     \begin{figure}
     \begin{minipage}{0.4\textwidth}
    \centering
       \include{Graphics/Lem10P3}
       \caption{}
       \label{Lem10P3}
    \end{minipage}%
    \begin{minipage}{0.4\textwidth}
    \centering
       \include{Graphics/Lem10P6}
       \caption{}
       \label{Lem10P6}
    \end{minipage}
    \end{figure}

      In Figure~\ref{Lem10P1} and Figure~\ref{Lem10P2} we see that the 4 internal vertices of the partial configurations constitute 4-cycles contained in $X$. However, there is no way to triangulate these 4-cycles so that all the internal vertices have nonpositive curvature so we get contradictions. The partial diagram in \ref{Lem10P3} is part of a unique minimal filling diagram shown in Figure~\ref{Lem10P6}. Up to rotation this filling diagram can be labeled as in Figure~\ref{Lem10P9}. We see that $X$ contains the 4-cycle $tx,tsv,z,tstsv$ which implies that $X$ contains $(tx,z)$. Therefore $X$ contains the 6-cycle $\gamma = tx,z,y,x,tz,tw$ show in Figure~\ref{Lem10B3}. By Lemma~\ref{D8L0} we see that there is a vertex adjacent to all the vertices in $\gamma$ or $X$ contains a diagonal of $\gamma$.

     \begin{figure}
     \begin{minipage}{0.4\textwidth}
    \centering
       \include{Graphics/Lem10P9}
       \caption{}
       \label{Lem10P9}
    \end{minipage}%
    \begin{minipage}{0.4\textwidth}
    \centering
       \include{Graphics/Lem10B3}
       \caption{}
       \label{Lem10B3}
    \end{minipage}
    \end{figure}

     First let us assume that there is a single vertex $a$ adjacent to every vertex in $\gamma$. Then $X$ contains the 5-cycle $a,tx,tv,v,x$ shown in Figure~\ref{Lem10B1}. If $X$ contains $(x,tv)$ or $(tx,v)$ then $x$ is adjacent to 5 consecutive vertices in the 10 cycle, so assume not. Then $X$ contains $(a,v)$ and $(a,tv)$. By a symmetric argument $X$ contains $(a,ststv)$ and $(a,ststsv)$. Then $X$ contains the 4-cycle $a,tv,tsv,z$ as seen in Figure~\ref{Lem10B2}. If $X$ contains $(z,tv)$ then $z$ is adjacent to 5 consecutive vertices so assume not. Then $X$ contains $(a,tsv)$. Then $X$ contains the 5-cycle $a,tsv,tstv,tstsv,ststsv$ and since $X$ does not contain any diagonals of the 10-cycle we have that $X$ contains $(a,tstv)$ and $(a,tstsv)$. Therefore $a$ is adjacent to more than 5 consecutive vertices of the 10-cycle so we are done.

     \begin{figure}
     \begin{minipage}{0.4\textwidth}
    \centering
       \include{Graphics/Lem10B1}
       \caption{}
       \label{Lem10B1}
    \end{minipage}%
    \begin{minipage}{0.4\textwidth}
    \centering
       \include{Graphics/Lem10B2}
       \caption{}
       \label{Lem10B2}
    \end{minipage}
    \end{figure}

     Now assume that instead $X$ contains some diagonal of $\gamma$. If $X$ contains the diagonal $(x,tx)$ then $X$ contains the 4-cycle $x,tx,tv,v$. Both of its diagonals produce a vertex adjacent to at least 5 consecutive vertices. By a symmetric argument if $X$ contains $(tz,z)$ then we can produce a vertex adjacent to at least 5 consecutive vertices, so assume $X$ contains neither $(x,tx)$ nor $(tz,z)$. If $X$ contains $(x,z)$ then $X$ contains the 4-cycle $x,z,tx,tz$ which has diagonals $(x,tx)$ and $(z,tz)$, which are by assumption not contained in $X$, so $X$ does not contain $(x,z)$. Now with our assumptions $X$ has to contain $(tw,y)$. So $X$ contains the 5-cycle $tw,sv,stv,stsv, y$ shown in Figure~\ref{Lem10B4}. If either $tw$ or $y$ are adjacent to two additional vertices in the 10-cycle we are done and if not $X$ contains $(tw,stv)$ and $(y,stv)$. Then $X$ contains the 5-cycle $v,w,y,stv,sv$ shown in Figure~\ref{Lem10B5}. If $X$ contains $(y,sv)$ then $y$ is adjacent to 5 consecutive vertices and if not then $X$ contains $(w,sv)$ and $(w,stv)$ making $w$ adjacent to 5 consecutive vertices in the 10-cycle, so we are done.

        \begin{figure}
     \begin{minipage}{0.4\textwidth}
    \centering
       \include{Graphics/Lem10B4}
       \caption{}
       \label{Lem10B4}
    \end{minipage}%
    \begin{minipage}{0.4\textwidth}
    \centering
       \include{Graphics/Lem10B5}
       \caption{}
       \label{Lem10B5}
    \end{minipage}
    \end{figure}

     Now assume there is some vertex $w$ adjacent to $5$ consecutive vertices of the 10-cycle as in Figure~\ref{Lem10B7}. Then $X$ contains the 4-cycle $w, sv, stw, tv$ show in Figure~\ref{Lem10B8}. Therefore $X$ contains $(w,stw)$. Therefore $X$ contains the 5-cycle $w,stw,(st)^2w, (st)^3w,(st)^4w$ (see Figure~\ref{Lem10B9}). Therefore $X$ contains $(w,(st)^2w)$. Then $X$ contains the 5-cycle $w,(st)^2w, ststsv,tstsv,tstv$, so $X$ contains one of $(w,tstsv)$ or $((st)^2w, tstsv)$ showing that $X$ contains a vertex adjacent to 6 consecutive vertices of the 10-cycle.

        \begin{figure}
     \begin{minipage}{0.4\textwidth}
    \centering
       \include{Graphics/Lem10B7}
       \caption{}
       \label{Lem10B7}
    \end{minipage}%
    \begin{minipage}{0.4\textwidth}
    \centering
       \include{Graphics/Lem10B8}
       \caption{}
       \label{Lem10B8}
    \end{minipage}%
    \begin{minipage}{0.4\textwidth}
    \centering
       \include{Graphics/Lem10B9}
       \caption{}
       \label{Lem10B9}
    \end{minipage}
    \end{figure}

     Now assume a vertex $w$ is adjacent to 6 consecutive vertices as in Figure~\ref{Lem10B10}. $X$ contains the 4-cycle $w,stv,stw,ststv$ therefore $X$ contains $(w,stw)$. Therefore $X$ contains the 5-cycle $w,stw,(st)^2w, (st)^3w, (st^4)w, (st)^5w$ so $X$ contains $(w, (st)^2 w)$. Therefore $X$ contains the 4-cycle $w, (st)^2w,tsv,tv$ shown in Figure~\ref{Lem10B12}. Therefore $w$ is adjacent to 7 consecutive vertices of the 10-cycle.

      \begin{figure}
     \begin{minipage}{0.4\textwidth}
    \centering
       \include{Graphics/Lem10B10}
       \caption{}
       \label{Lem10B10}
    \end{minipage}%
    \begin{minipage}{0.4\textwidth}
    \centering
       \include{Graphics/Lem10B12}
       \caption{}
       \label{Lem10B12}
    \end{minipage}
    \end{figure}

     Now assume a vertex $w$ is adjacent to 7 consecutive vertices. The 4-cycle $w,sv,tw,tsv,w$ shows that $X$ contains $(w,tw)$. Then $X$ contains the 4-cycle $w,tw,ststsv,ststv$ shown in Figure~\ref{Lem10B13}. Therefore $w$ is adjacent to 8 consecutive vertices.

      \begin{figure}
     \begin{minipage}{0.4\textwidth}
    \centering
       \include{Graphics/Lem10B13}
       \caption{}
       \label{Lem10B13}
    \end{minipage}%
    \begin{minipage}{0.4\textwidth}
    \centering
       \include{Graphics/Lem10B14}
       \caption{}
       \label{Lem10B14}
    \end{minipage}
    \end{figure}

     Now assume a vertex $w$ is adjacent to 8 consecutive vertices as in Figure~\ref{Lem10B14}. Then $w,tsv,tstv,tstsv,ststsv$ is a 5-cycle, therefore $w$ is adjacent to every vertex of the 10-cycle.

    \end{proof}

 \begin{lemma}
     Suppose $X$ is a systolic complex, which is acted on geometrically by $W$. Let $v \in X^0$ such that $\langle r, s \rangle v$ and $ \langle t, r \rangle v$ are cliques. Then either $\langle s, t \rangle v$ is a clique or there is some other vertex $v'$ such that $\langle r, s \rangle v', \langle t,r \rangle v',$ and $\langle s, t \rangle v'$ are cliques.

     \label{lem:wr->cliques}
 \end{lemma}

 \begin{proof}

 If $X$ contains any diagonal of $\langle s, t \rangle v$ then by Lemma \ref{lem:10cycle-diag->clique} it is clique, so assume not. Then by Lemma~\ref{lem:10-cyc-1Int} there is a vertex $w$ adjacent to all the vertices of $\langle s,t \rangle v$. Then $X$ contains the 4-cycles $w,tv,tsrv,tstv$ and $w,sv,rv,tv$ shown in Figure~\ref{10-cycA1}, which show that $X$ contains $(w,tsrv)$ and $(w,rv)$ respectively. Therefore $X$ contains the 5-cycle $w,rv,rw, trsv,tsrv$ seen in Figure~\ref{10-cycA2}. By assumption $X$ does not contain $(rv,trsv)$, therefore $X$ contains $(w,rw)$ or $(rv,tsrv)$. 

 \begin{figure} [!ht]
        \center
        \begin{minipage}{0.6\textwidth}
        \center
                \include{Graphics/10-cycA1}
                \caption{}
                \label{10-cycA1}
            \end{minipage}%
            \begin{minipage}{0.6\textwidth}
            \center
                 \include{Graphics/10-cycA2}
                \caption{}
                \label{10-cycA2}
            \end{minipage}
    \end{figure}

 First we deal with the case where $X$ contains $(rv,tsrv)$. Then we claim $v'=rv$ satisfies the conditions of the lemma. It is already clear that $\langle r,s \rangle v'$ and $ \langle t,r \rangle v'$ are cliques. Then since $(rv,tsrv)$ is a diagonal of the 10-cycle $\langle s,t \rangle v'$ by Lemma~\ref{lem:D8_diag->clique} $\langle s,t \rangle v'$ is a clique.

 Now assume that $X$ does not contain $(rv,tsrv)$, then it needs to contain $(w,rw)$. We claim that $v'=w$ satisfies the conditions of the lemma. If $w_1$ and $w_2$ are two vertices of $\langle s,t \rangle v'$, then since $s$ and $t$ fix $\langle s, t \rangle v $ we have that $X$ contains the 4-cycle $w_1,v,w_2,stv $, which shows that $X$ contains $(w_1,w_2)$. Therefore $\langle s,t \rangle v'$ is a clique. We can reuse the argument that showed that $X$ contained $(w,rw)$ to show that $X$ contains $(sw,rw)$, which is a diagonal of the 8-cycle $\langle r, s \rangle v'$ so by Lemma~\ref{lem:D8_diag->clique} $\langle r, s \rangle v'$ is a clique. Then since $\langle t, r \rangle v'$ is trivially a clique we are done. 
 \end{proof}

Now we are ready to prove the main result of this section Proposition~\ref{prop:base}.

\begin{proof}[Proof of Proposition~\ref{prop:base}]

Let $v$ and $w$ be vetices of $A$ and $B$, respectively, such that $d(v,w) = d(A,B)$. Using the group action we infer the configuration of edges in $X$ shown in Figure~\ref{octand10}. 

        \begin{figure} [!ht]
            \include{Graphics/OctAnd10Gon}
            \caption{}
            \label{octand10}
        \end{figure}

        % \begin{claim}
        %     There is a 1-cell with vertices $v$ and $tv$ or there is a 1-cell with vertices $w$ and $rw$ or there exist $v'$ such that $\langle s,t \rangle v', \langle t , r \rangle v$ and $\langle r,s \rangle$ are cliques.
        % \end{claim}

        % \begin{proof}[Proof of Claim]
            Let $C$ be the $8$-cycle $v,w,tw,tv,trv,trw,rw,rv$ as seen in Figure \ref{octand10}. We label the edge types $(v,rw)$ and $(v,tw)$ by {\bf 2} and {\bf 2'}. We label the edge types $(v,trv)$ and $(w,trw)$ by {\bf 4} and {\bf 4'}. 

            We proceed by proving a collection of implications involving conditions on $X$. The implications are summarized in Figure~\ref{proofDiag}. The condition given by $P_{8}$ in the figure is that $C$ has a filling diagram with  domain $P_8$. The conditions given by edge labels are that $X$ contains those edges. The arrows indicate that every condition implies that $X$ satisfies some condition further to the right in the figure (not necessarily the condition at the head of an arrow).

             By Lemma~\ref{D8L1} there exists a minimal filling diagram $f:D \rightarrow X$ of $C$, such that $D = P_8$, $D= P_{10}$, or there is a vertex on the boundary of $D$ with boundary curvature $2$. By Lemma~\ref{lem:8cyl:P8->P10} $D \neq P_{10}$. Therefore in Figure~\ref{proofDiag} we know that either $X$ satisfies $P_8$ or $2$ or $2'$. Therefore it is enough to show that every condition implies that $X$ satisfies a condition further to the right to show the proposition. 

            \begin{figure}[!ht]
            \[
                \mathbf{2} \text{ or } \mathbf{2}' \rightarrow P_8 \rightarrow \mathbf{2} \text{ and } \mathbf{2}' \rightarrow \mathbf{4} \text{ or } \mathbf{4}' \rightarrow (v,tv) \text{ or } (w,rw) \rightarrow \text{Prop}~\ref{prop:base}
            \]
            % \begin{tikzpicture}
            %      \matrix (m) [matrix of math nodes, row sep=2em,column sep=2em]
            %      {
            %         & {\bf 2} &  & & & &\\
            %         P_{10} & P_8 & {\bf 2} \text{ and } {\bf 2'} & {\bf 4} \text{ or } {\bf 4'}& (v,tv) \text{ or } (w,rw) & \text{Prop~\ref{prop:base}} \\ 
            %          & {\bf 2'} & & & &\\
            %      };
            %      \path[-stealth]
            %      (m-1-2) edge (m-2-2)
            %      (m-2-1) edge (m-2-2)
            %      (m-2-2) edge (m-2-3)
            %      (m-2-3) edge (m-2-4)
            %      (m-2-4) edge (m-2-5)
            %      (m-2-5) edge (m-2-6)
            %      (m-2-4) edge (m-2-5)
            %      (m-3-2) edge (m-2-2)

            %      ;
            %  \end{tikzpicture}
             \caption{}
             \label{proofDiag}
             \end{figure}

              {\bf 2} or {\bf 2$'$} $\rightarrow P_8$:

        Assume $X$ contains {\bf 2$'$}, the case where $X$ contains {\bf 2} is handled similarly. We see that $X$ contains the 6-cycle $v,tw, tv, trv, rw, rv$. By Lemma~\ref{D8L0} either there is a vertex $x$ adjacent to every vertex of the 6-cycle or $X$ contains some diagonal of the 6-cycle.

         First assume $X$ contains some diagonal of $C$. Any diagonal of $C$ splits $C$ into cycles of length less than $5$, which need to have triangulations since $X$ is systolic, so $C$ has some triangulation in $X$.  If $X$ contains $(v,rw)$ or $(tw,trw)$ then $X$ contains {\bf 2} and {\bf 2}'  so assume not.  If $X$ contains $(rv,tv)$ or $(rv,trv)$ then we are done, so assume $X$ does not contain these edges. These assumption are shown in Figure \ref{6filling}. This allows for only one valid triangulation shown in Figure \ref{6filling2}. Therefore $X$ contains $(rw,tw)$ in edge type {\bf 4}$'$.

        \begin{figure} [!ht]
        \center
        \begin{minipage}{0.3\textwidth}
        \center
                \include{Graphics/6cyclefilling}
                \caption{}
                \label{6filling}
            \end{minipage}%
            \begin{minipage}{0.3\textwidth}
            \center
                 \include{Graphics/6cyclefilling2}
                \caption{}
                \label{6filling2}
            \end{minipage}%
            \begin{minipage}{0.3\textwidth}
        \center
                \include{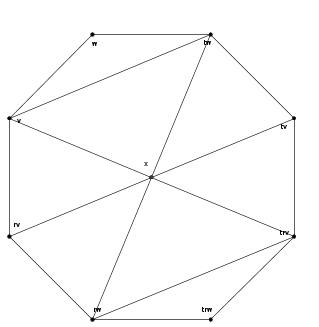}
                \caption{}
                \label{6filling3}
            \end{minipage}
            \end{figure}

    Now assume there is a vertex $x$ adjacent to all of the vertices in the 6-cycle as in Figure~\ref{6filling3}. The 4-cycle $x,v,w,tv$ has diagonals $(w,x)$ and $(v,tv)$. If $X$ contains $(v,tv)$ then the claim is true, so assume $X$ instead contains $(w,x)$. The 4-cycle $x,trv,trw, rv$ has diagonals $(rv,trv)$ and $(x,trw)$. If $X$ contains $(rv,trv)$ then we are done, so assume not, then $X$ contains $(x,trw)$. Now $x$ is adjacent to every vertex of $C$ so $P_8$ is a filling diagram of $C$.

$P_8$ $\rightarrow$ {\bf 2} and {\bf 2$'$}:

 Assume $X$ does not contain $(sx,x), (sv, tsv),$ and $(sw,rsw)$. First assume for contradiction that $X$ does not contain {\bf 2}.  $X$ contains the 4-cycle $x,rv,sv, w$ shown in Figure~\ref{1InternalVertB1}. By assumption $X$ does not contain {\bf 2}, so $X$ must contain $(x,sv)$. Therefore $X$ contains the 5-cycle $x,sv,sw,tsw,tsv$ as seen in Figure~\ref{1InternalVertB2}. By assumption $(sv,tsv)$ is not contained in $X$, so $X$ has to contain either $(x,sw)$ or $(x,tsw)$. Assume without loss of generality that $X$ contains $(x,sw)$, if not then we can let $x'=sx$ and then apply the subsequent argument to $x'$. Therefore $X$ contains the 5-cycle $x,sw,sx,srv,rv$ shown in Figure~\ref{1InternalVertB3}. By assumption $X$ does not contain  the diagonals $(sw, srv)$ and $(x,sx)$. Therefore $X$ must contain the diagonal $(sw,rv)$. Therefore $X$ contains the 4-cycle $rv,sw,w,srv$ which has diagonals solely in {\bf 2} contradicting our assumption. Therefore $X$ contains {\bf 2}. To show that $X$ contains {\bf 2}$'$ we use a symmetric argument by swapping the labels (ie replace $v$ by $w$).

 \begin{figure} [!tbp]
    \centering
    \begin{minipage}{0.5\textwidth}
    \centering
                \include{Graphics/1InternalVertexOctB1}
                \caption{}
                \label{1InternalVertB1}
    \end{minipage}%
    \begin{minipage}{0.5\textwidth}
    \centering
                \include{Graphics/1InternalVertexOctB2}
                \caption{}
                \label{1InternalVertB2}
    \end{minipage}
    \end{figure}

    \begin{figure}[!tbp]
    \centering
                \include{Graphics/1InternalVertexOctB3}
                \caption{}
                \label{1InternalVertB3}
    \end{figure}

    % \begin{figure} [!tbp]
    % \centering
    % \begin{minipage}{0.5\textwidth}
    % \centering
    %             \include{Graphics/1InternalVertexOct}
    %             \caption{}
    %             \label{1InternalVertB3}
    % \end{minipage}%
    % \begin{minipage}{0.5\textwidth}
    % \centering
    %             \include{Graphics/1InternalVertexOct3}
    %             \caption{}
    %             \label{1InternalVertB4}
    % \end{minipage}
    % \end{figure}

Now we deal with the case where $X$ contains $(x,sx)$. We claim that $A' = \langle s,t \rangle x$ and $B' = \langle r, s \rangle x$ satisfy the proposition. To see that $\langle t, r \rangle x$ is a clique observe that $r$ and $t$ fix $C$ so every point in $\langle r,t \rangle x$ is adjacent to every point in $C$. Therefore if $v_1, v_2 \in \langle r, t \rangle x $ then $v_1,v,v_2,rw$ is a 4-cycle and by assumption $X$ does not contain $(v,rw)$ so $v_1$ and $v_2$ are adjacent. Therefore $\langle r, t \rangle x$ is a clique. The points $rx$ and $tx$ are adjacent to $sx$ by the same argument that was used to show $x$ was adjacent to $sx$. Therefore by Lemma~\ref{lem:D8_diag->clique} and Lemma~\ref{lem:10cycle-diag->clique} the sets $\langle t , r \rangle x$ and $\langle r,s \rangle x$ are cliques. 

If $X$ contains $(sv,tsv)$, then we label $sv$ as $v'$. Therefore $\langle r, s \rangle v'$ is a clique fixed by $\langle r , s \rangle$ at distance $1$ from $\langle s, t \rangle w$. We see that $X$ contains the edge $(v', tv')$, so we let $v'$ take the place for $v$ and we are done. A symmetric argument is used for if $X$ contains $(sw,rsw)$ by switching the labels.
            
    {\bf 2} and {\bf 2$'$} $\rightarrow$ {\bf 4} or {\bf 4$'$}:

        $X$ contains the 4-cycle $v, tw, trv, rw$ as shown in Figure~\ref{4Oct}. $X$ contains one of its diagonals $(v,trw)$ or $(rw,tw)$ both of which imply that $X$ contains an edge between 0-cells at distance 4 in $C$.

     {\bf 4} or {\bf 4$'$} $\rightarrow (v,tv) $ or $(w,rw)$:

             If $X$ contains $(v,trv)$ then $X$ contains the 4-cycle $v,trv,tv,rv$ shown in Figure~\ref{4Oct2} with diagonals in the edge type $(v,tv)$. Likewise if $X$ contains $(w,trw)$ then $X$ contains $(w,rw)$ by a similar argument.

            \begin{figure} [!ht]
        \center
            \begin{minipage}{0.45\textwidth}
            \center
                 \include{Graphics/4-cycleInOct}
                \caption{}
                \label{4Oct}
            \end{minipage}%
                 \begin{minipage}{0.45\textwidth}
        \center
                \include{Graphics/4-cycleInOct2}
                \caption{}
                \label{4Oct2}
            \end{minipage}
            \end{figure}

$(v,tv)$ or $(w,rw) \rightarrow $ Prop~\ref{prop:base}:

Follows directly from either Lemma~\ref{lem:tv->cliques} or Lemma~\ref{lem:wr->cliques}.

 % \end{proof}
 \end{proof}

 \section{Induction Step}
 \label{sec:ind}

In this section we prove Proposition~\ref{prop:ind}.

  \begin{proof}[Proof of Proposition~\ref{prop:ind}]

   We proceed by contradiction and assume that $n$ is the minimal distance between two cliques with the desired properties. Let $v\in A^0$ and $w \in B^0$ be chosen such that $d(v,w) = n$.

   \begin{claim}
      \label{cl-ex-P}
       There exists a minimal path $P$ from $v$ to $w$ such that for all adjacent 0-cells $a$ and $b$ in $P$ we have that $a,sa,b,sb$ forms a clique.
   \end{claim}

   \begin{proof}
        We inductively construct the path $P$ in the following way. We let $v_0 = v$. Now we construct $v_i$ for $1\leq i < n$. Let $Q$ be the edge spanned by $w$ and $sw$, which is a convex subcomplex of $X$. By the induction hypothesis we know that $v_{k-1}$ and $sv_{k-1}$ span a $1$--cell $\rho$ in $S_{n-(k-1)}(Q)$. By Lemma~\ref{lem:proj} $B_{n-k}(Q) \cap Res(\rho,X)$ is a nonempty simplex. Let $v_k$ be a 0-cell in this simplex. Since $sQ= Q$ and $s\rho = \rho$, we see that $sv_k \in B_{n-k}(Q) \cap Res(\rho,X)$, therefore $v_k$ and $sv_k$ are adjacent. Therefore $v_{k-1}, sv_{k-1}, sv_k, v_k$ is a 4-cycle, therefore $X$ contains $(v_{k-1}, sv_k)$ making $v_{k-1}, sv_{k-1}, sv_k, v_k$ a clique, thus finishing the proof.     
   \end{proof}

   \begin{claim}
           For any $P$ satisfying the conditions of Claim~\ref{cl-ex-P}, then $P$ is disjoint from $tP$ and $rP$.
       \end{claim}

       \begin{proof}
       Suppose not. In particular suppose that $rP$ intersects with $P$ (the case where $tP$ intersects $P$ is handled similarly). We pick a vertex $x \in P \cap rP$. Since $A = rA$ and the action of $W$ is distance preserving we see that $d(A,x) = d(A, rx)$. It is easy to see that $rx \in P$ as well therefore since $P$ is a minimal path it contains one vertex at distance $d(A,x) = d(A,rx)$ from $A$ therefore $x = rx$. 

        Therefore by Lemma~\ref{lem:D8_diag->clique} $\langle r, s \rangle x$ is a clique. This contradicts the minimality of $n$.   
       \end{proof}

       Let $P$ be a minimal path satisfying the conditions of Claim~\ref{cl-ex-P}.  We can concatenate the paths $P,tP, rP,$ and$ rP$ to form a cycle $\gamma$. Let $f:D \rightarrow X$ be a minimal filling diagram of $\gamma$, which we know exists by Theorem~\ref{thm:exist}. We say the \emph{total boundary curvature} of $f$ is the sum of the curvature of all its boundary 0-cells given explicitly by the formula

       \[
           \sum_{v \in \partial D} \kappa_\partial(v).
       \] We now need a claim about the structure of $\gamma$.

       We pick $A,B,v,w,P,$ then $f$ to minimize the number of $2-cells$ in $D$, and secondly to minimize the total boundary curvature of $f$.

      We can partition the vertices in $\partial D$ into $8$ sets. For $n$ odd the partitioning is shown in Figure~\ref{IndOdd} and for $n$ even the partition is shown Figure~\ref{IndEven}. We now make two claims concerning the sums of boundary curvatures of these sets, which we prove later in the section.

      \begin{claim}
      \label{cl:E}
          For all $1 \leq i \leq 4$, either the following inequality holds,
          \[
              \sum_{v \in E_i} \kappa_\partial(v) \leq 1
          \] or we can pick $A',B',v',w',P'$, then $f'$ that satisfy our minimality assumption such that the inequality holds and such that all vertices in $C$ have the same boundary curvature in $f$ as $f'$ except for $a$ and $d$ (labels shown in Figure~\ref{tempPict}) with the boundary curvature of $d$ being decreased by $1$ and the boundary curvature of $a$ being increased by $1$.
      \end{claim}

      \begin{figure}
          \centering
                \include{Graphics/tempPict}
                \caption{}
                \label{tempPict}    
      \end{figure}

      \begin{claim}
      \label{cl:F}
          For all $1 \leq i \leq 4$, the following inequality holds,
          \[
              \sum_{v \in F_i} \kappa_\partial(v) \leq 0.
          \]
      \end{claim}

      For $n \geq 3$ we can apply Claim~\ref{cl:E} for all $E_i$ and then apply Claim~\ref{cl:F} for all $F_i$ and we see that the total boundary curvature of the resultant filling diagram satisfies
      \[
           \sum_{v \in \partial D} \kappa_\partial(v) =  \sum_{i=1}^4 \sum_{v \in E_i} \kappa_\partial(v) + \sum_{i=1}^4 \sum_{v \in F_i} \kappa_\partial(v) \leq 4 \cdot 1 + 4 \cdot 0 < 6
       \]
       which contradicts Lemma~\ref{lem:gauss}.

       For $n=3$ the $F_i$ are empty, so when we apply  Claim~\ref{cl:E} to $E_i$ in the resulting $f$  the boundary curvature of a vertex in $E_{i+1}$ (with the index taken mod 4) may be increased by 1. To deal with this obstruction we iteratively apply Claim~\ref{cl:E} to $E_1$ then $E_2$ then $E_3$ giving us a filling diagram $f$. We also observe that since in Claim~\ref{cl:E} we only decrease the curvature by $1$ of a single vertex in the new filling diagram, we observe that the sum of boundary curvatures of vertices in $E_4$ is at most $2$. Therefore the choice of $f$ satisfies
       \[
           \sum_{v \in \partial D} \kappa_\partial(v) =  \sum_{i=1}^3 \sum_{v \in E_i} \kappa_\partial(v) + \sum_{v \in E_4}\kappa_\partial(v) \leq  3 \cdot 1 + 2=  5 < 6,
       \] which contradicts Lemma~\ref{lem:gauss}.

\begin{figure} [!tbp]
    \centering
    \begin{minipage}{0.5\textwidth}
    \centering
                \include{Graphics/IndEven}
                \caption{}
                \label{IndEven}
    \end{minipage}~
    \begin{minipage}{0.5\textwidth}
    \centering
                \include{Graphics/IndOdd}
                \caption{}
                \label{IndOdd}
    \end{minipage}
    \end{figure} 

    The next few claims are concerned with proving Claim~\ref{cl:F}.

      \begin{claim}
      \label{Cl-curv1}
          If $a,b,c \in P$ are 0-cells distinct from $v$ and $w$ with a common neighbor $d \in D$. Then $X$ cannot contain $(d,sa)$ as shown in Figure~\ref{LemCurv1-1}. This claim is true for any minimal filling diagram $f$, not necessarily satisfying our assumption to minimize the total boundary curvature.
      \end{claim}

    \begin{figure}[!tbp]
    % \centering
    % \begin{minipage}{0.5\textwidth}
    \centering
                \include{Graphics/LemCurv1-1}
                \caption{}
                \label{LemCurv1-1}
    % \end{minipage}
    \end{figure}

    \begin{proof}
        If $X$ contains the edge $(c,sa)$ or $(sc,sa)$ then our minimality assumption is contradicted. If $X$ contains $(d,sd)$ then we can replace $b$ in the path $P$ by $d$ yielding a new path $P'$ and a new filling diagram $f'$ with fewer 2 cells, contradicting our minimality assumption. Therefore there is no valid triangulation of the 5-cycle $d,sa,sd,sc,c$, which gives a contradiction.
    \end{proof}

      \begin{claim}
          There are no two consecutive vertices in the interior of $P, tP, rP,$ or $trP$ with boundary curvatures $1$ and $0$. In particular this shows that any consecutive set of vertices in the interior of $P, tP, rP$ or $trP$ of an even cardinality has at most curvature 0, thus proving Claim~\ref{cl:F}.
      \end{claim}

       \begin{figure} [!tbp]
    \centering
    \begin{minipage}{0.5\textwidth}
    \centering
                \include{Graphics/side-_1,0_}
                \caption{}
                \label{side-(1,0)}
    \end{minipage}
    \end{figure}

      \begin{proof}
        Assume there are two consecutive vertices with curvature $1$ and $0$ as in Figure~\ref{side-(1,0)}.  By Claim~\ref{Cl-curv1} we know that $X$ cannot contain $(a,sw_2)$. If $X$ contains $(w_3,w_1)$ then $P$ is not a minimal path between $v$ and $w$, contradicting our assumption. Since $X$ contains the 5-cycle $a,w_1,sw_2,w_3,b$ and does not contain $(a, sw_2)$ or $(w_3,w_1)$, $X$ must contain $(b,sw_2)$. $X$ contains the 4-cycle $a,w_2,sw_1,w_0$ so must contain $(a,sw_1)$. Therefore $X$ contains the 4-cycles $b,sw_2,w_1,a$ and $b,sw_2,sw_1,a$, which imply that $X$ contains $(b,w_1)$ and $(b,sw_1)$ respectively. Therefore $X$ contains the 5-cycle $b,sw_1,sb, sw_3,w_3$, which implies the that $X$ contains $(b,sb)$. Then we can replace $w_2$ in $P$ with $P$ yielding a new path $P'$ and a filling diagram $f'$ with fewer 2-cells, contradicting our minimality assumption.       
      \end{proof}

      The next few claims are used to show Claim~\ref{cl:E}. In the claims we look at the possible configurations of the boundary curvatures of the vertices of an $E_i$ that sum to at least $2$. In each case we show that the resulting filling diagram satisfies the conditions of Claim~\ref{cl:E}. We now list the possible assignments of boundary curvature of $b,c,$ and $d$ (with labels taken from Figure~\ref{tempPict}) where the sum of their curvatures is at least $2$. We give each possible assignment as a string of digits with the digits refering to the boundary curvature of $b,c, $ and $d$ respectively. The configurations are: $111, 110, 101, 011, 211, 210, 201, 200, 21-1,2-11, 120,020,$ and $12-1$.

       \begin{remark}
        There is no filling diagram that satisifies our minimality assumption such the curvatures of $b,c,$ and $d$ are given by any string of the form $*21$, where $*$ represents any digit. 
    \end{remark}

      \begin{claim}
          For this following claim we do not assume that $f$ was chosen to minimize total boundary curvature. The vertices $rv,v,$and $ b$, do not have curvature $1,1$ and $1$ as seen in Figure~\ref{Ind-(1,1)1-1}. 
          \label{1,1,1}
      \end{claim}

       \begin{figure} [!tbp]
    \centering
    \begin{minipage}{0.5\textwidth}
    \centering
                \include{Graphics/Ind__1,1_1-1}
                \caption{}
                \label{Ind-(1,1)1-1}
    \end{minipage}
    \end{figure}

      \begin{proof}
          First assume that $X$ does not contain $(rb,b)$. We see that $X$ contains the 5-cycle $b,a,rb,rsv,sv$. $X$ cannot contain $(a,sv)$ by Claim~\ref{Cl-curv1}. So $X$ must contain $(rsv,b)$. Then $X$ contains the 4-cycle $rb,sv,b,a$, but $X$ does not contain a triangulation of this 4-cycle so we have a contradiction.

          Now we assume that $X$ contains $(rb,b)$. Then $X$ contains the 4-cycle $rb,b,sv,rsv$, so $X$ contains $(rb,sv)$. Also $X$ contains the 4-cycle $(a,v,sb,c)$ and since $X$ cannot contain $(v,c)$ by the minimality of $P$ we see that $X$ contains $(a,sb)$. Therefore $X$ contains the 4-cycle $a,sb,sv,rb$, so $X$ must contain $(rb,sb)$. The edge $(rb,sb)$ is a diagonal of the 8-cycle $\langle r, s \rangle b$, so by Lemma $\ref{lem:D8_diag->clique}$, we see that $\langle r, s \rangle b$ is a clique. Then the fact that $b$ is closer to $w$ than $v$ contradicts the minimality of $P$.
      \end{proof}

      \begin{remark}
        The previous claim still holds if instead of focusing on the boundary curvatures of $v,rv,$ we instead worked with the 0-cells $w$ and $tw$. The proof is identical aside from the fact that instead of applying  Lemma~\ref{lem:D8_diag->clique} to show the 8-cycle is a clique, we would instead apply Lemma~\ref{lem:10cycle-diag->clique} to show a 10-cycle is a clique. Therefore we can apply this claim to any $E_i$. This remark is true for the subsequent claims as well.
    \end{remark}
    
    \begin{claim}
        The 0-cells $rv,v,b$ cannot have boundary curvatures $1,1,$ and $0$ as in Figure~\ref{Ind_0(1,1)0-1}.
        \label{(1,1)0}
    \end{claim}

     \begin{figure} [!tbp]
    \centering
    \begin{minipage}{0.5\textwidth}
    \centering
                \include{Graphics/Ind_0_1,1_0-1}
                \caption{}
                \label{Ind_0(1,1)0-1}
    \end{minipage}~
    \begin{minipage}{0.5\textwidth}
    \centering
                \include{Graphics/Ind_0_1,1_0-3}
                \caption{}
                \label{Ind_0(1,1)0-3}
    \end{minipage}
    \end{figure}

    \begin{proof}

    First we observe that $X$ cannot contain the edge $(rv,b)$ because if we replace $(a,v)$ with $(rv,b)$ we obtain a filling diagram with a smaller total boundary curvature contradicting our minimality assumption.

    We first deal with the case where $X$ contains the edge $(a,sb)$. Then $X$ contains the 5-cycle $a,sb,sa,srv,rv$ shown in Figure~\ref{Ind_0(1,1)0-3}. Since $X$ does not contain $(rv,b) = (sb,srv)$ either $X$ contains $(a,sa)$ or $X$ contains $(sa,rv)$. However if $X$ contains $(sa,rv)$, then it contains the 4-cycle $sa,b,a,rv$ shown in Figure~\ref{Ind_0(1,1)0-4}, which implies that $X$ contains $(a,sa)$. Therefore $X$ must contain $(a,sa)$. 

     \begin{figure} [!tbp]
    \centering
    \begin{minipage}{0.5\textwidth}
    \centering
                \include{Graphics/Ind_0_1,1_0-4}
                \caption{}
                \label{Ind_0(1,1)0-4}
    \end{minipage}~
    \begin{minipage}{0.5\textwidth}
    \centering
                \include{Graphics/Ind_0_1,1_0-2}
                \caption{}
                \label{Ind_0(1,1)0-2}
    \end{minipage}
    \end{figure}

    Next we observe that since $X$ contains the 4-cycle $a,rv,ra,b$, we get that $X$ contains $(a,ra)$. Therefore $\langle r,s \rangle a$ is a 8-cycle. The argument we used to show that $a$ is adjacent to $sa$ did not use the fact that $sa$ is adjacent to $sc$, therefore we can use this argument to show that $a$ is adjacent to $sra$ as well. Therefore by Lemma~\ref{lem:D8_diag->clique} $\langle r,s \rangle a$ is a clique  Now we can replace the clique $A$ by the clique $\langle r, s \rangle a$ and the new cycle $\gamma$ allows for a filling diagram with fewer 2-cells yielding a contradiction.

    \begin{figure} [!tbp]
    \centering
    \begin{minipage}{0.5\textwidth}
    \centering
                \include{Graphics/Ind_0_1,1_0-5}
                \caption{}
                \label{Ind_0(1,1)0-5}
    \end{minipage}~
    \begin{minipage}{0.5\textwidth}
    \centering
                \include{Graphics/Ind_0_1,1_0-7}
                \caption{}
                \label{Ind_0(1,1)0-7}
    \end{minipage}
    \end{figure}

    Now assume instead $X$ does not contain $(a,sb)$. $X$ contains the 4-cycle $rv,a,b,sv$ seen in Figure~\ref{Ind_0(1,1)0-2}. Therefore $X$ must contain $(a,sv)$. Next we observe that $X$ contains the 5-cycle $a,v,sb,d,c$ shown in Figure~\ref{Ind_0(1,1)0-5}. Since $X$ does not contain $(a,sb)$ and since $X$ does not contain $(v,d)$ by the minimality of $P$, we see that $X$ must contain $(v,c)$. $X$ also contains the 5-cycle $a,sv,sb,d,c$, which by a similar argument shows that $X$ contains $(sv,c)$. Therefore $X$ contains the 5-cycle $c,v,sc,sd,d$ shown in Figure~\ref{Ind_0(1,1)0-7}. By the minimality of $P$, $X$ cannot contain $(v,d)$ or $(v,sd)$. Therefore $X$ must contain $(c,sc)$, but then we can replace $b$ in $P$ with $c$ producing a new cycle $\gamma$ admitting a filling diagram with fewer interior 2-cells, contradicting our minimality assumption. 
    \end{proof}

    \begin{claim}
        The 0-cells $rv,v,$ and $b$ cannot have curvatures $0,2$ and $0$ as seen in Figure~\ref{Ind_(0,2)0-1}.
        \label{(0,2)0}
    \end{claim}

     \begin{figure} [!tbp]
    \centering
    \begin{minipage}{0.5\textwidth}
    \centering
                \include{Graphics/Ind__0,2_0-1}
                \caption{}
                \label{Ind_(0,2)0-1}
    \end{minipage}~
    \begin{minipage}{0.5\textwidth}
    \centering
                \include{Graphics/Ind__0,2_0-2}
                \caption{}
                \label{Ind_(0,2)0-2}
    \end{minipage}
    \end{figure}

    \begin{proof}
        $X$ contains the 4-cycle $rb,a,b,v$ shown in Figure~\ref{Ind_(0,2)0-2}. If $X$ contains the edge $(a,v)$, then the 0-cells $rv$,$v$ and $b$ have curvature 1, so we apply the result from Claim~\ref{1,1,1} to get a contradiction. Otherwise $X$ contains $(rb,b)$. However if we replace $(a,rv)$ in the filling diagram with $(rb,b)$ we decrease the curvature of the boundary vertices by 1, contradicting that the filling diagram was constructed to minimize the boundary curvature.
    \end{proof}

    \begin{claim}
        The 0-cells $rv, v$ and $b$ cannot have curvatures $2, 0,$ and $1$ as in Figure~\ref{Ind_1(0,2)-1}.

        \label{Cl:1(0,2)}
    \end{claim}

     \begin{figure} [!tbp]
    \centering
    \begin{minipage}{1\textwidth}
    \centering
                \include{Graphics/Ind_1_0,2_-1}
                \caption{}
                \label{Ind_1(0,2)-1}
    \end{minipage}
    \end{figure}

    \begin{proof}
        If $X$ contains $(a,rv)$ then $rv,v,$ and $b$ have curvature 1, which contradicts Claim~\ref{1,1,1}. If $X$ contains $(rb,b)$ then we can replace $(a,rv)$ in the filling diagram $f$ with $(rb,b)$ reducing the total boundary curvature giving a contradiction. Therefore there is no triangulation of the 4-cycle $a,b,rv,rb$, proving the claim.
    \end{proof}

    \begin{claim}
        If the 0-cells $rv,v,b$ cannot have curvatures $2,0$ and $0$ as seen in Figure~ \ref{Ind_(2,0)0-1} then there is a minimal filling diagram of $\gamma$ that satisfies our minimality assumption such that $rv,v,b$ have curvatures 0, 2, and -1, the curvature of $rb$ is increased by 1, and the other vertices in $\gamma$ have the same curvature as in $f$. 
        \label{Cl:(2,0)0}

    \end{claim}

     \begin{figure} [!tbp]
    \centering
                \include{Graphics/Ind__2,0_0-1}
                \caption{}
                \label{Ind_(2,0)0-1}
    \end{figure}   

    \begin{proof}
       We see that $X$ contains the 4-cycle $rb,rv,b,a$. $X$ cannot contain $(rb,b)$ because this would contradict our minimality assumption, therefore $X$ contains $(rv,a)$. Therefore we construct $f'$ an alternate filling diagram by removing $(rb,v)$ and $(a,v)$ from $f$ and replacing them by the 1-cells $(a,rv)$ and $(rv,b)$. $f'$ satisfies the conditions of claim so we are done.
    \end{proof}

    \begin{claim}
        The 0-cells $v$ and $c$ cannot have curvatures $1$ and $1$ as in Figure~\ref{Ind_(0,1)1-1}. As consequence the 0-cells $rv, v$ and $c$ cannot have curvatures $0,1$ and $1$ or $2,1,$ and $1$.
        \label{(0,1)1}
    \end{claim}

      \begin{figure} [!tbp]
    \centering
    \begin{minipage}{0.5\textwidth}
    \centering
                \include{Graphics/Ind__0,1_1-1}
                \caption{}
                \label{Ind_(0,1)1-1}
    \end{minipage}~
    \begin{minipage}{0.5\textwidth}
    \centering
                \include{Graphics/Ind__0,1_1-2}
                \caption{}
                \label{Ind_(0,1)1-2}
    \end{minipage}
    \end{figure}   

    \begin{proof}
        First we notice that $X$ cannot contain $(rv,c)$ since if it did we could replace the 1-cell $(v,a)$ with $(rv,c)$ in the filling diagram to produce a filling diagram with less curvature in the boundary producing a contradiction. We see that $X$ contains the 4-cycle $rv,a,c,sv$ shown in Figure~\ref{Ind_(0,1)1-2}, so $X$ must contain $(a,sv)$, but this gives a contradiction by Claim~\ref{Cl-curv1}.
    \end{proof}

    \begin{claim}
        The 0-cells $rv, v$ and $c$ cannot have curvatures $1,2,$ and $0$ shown in Figure~\ref{Ind_(1,2)0-1}.
    \end{claim}

    \begin{figure}
        \include{Graphics/Ind__1,2_0-1}
        \caption{}
        \label{Ind_(1,2)0-1}
    \end{figure}

    \begin{proof}
        We can replace the 1-cells with ends $c$ and $rv$ in $f$ with the $1$-cell with ends $v$ and $rc$. In the resulting filling diagram $rv,v$ and $c$ have curvatures $2,1,$ and $1$ contradicting Claim~\ref{(0,1)1}.  
    \end{proof}

    \begin{claim}
        The 0-cells $rv, v$ and $c$ cannot have curvatures $1,0$ and $1$ as in Figure~\ref{Ind_(1,0)1-1}.  
    \end{claim}

     \begin{figure} [!tbp]
    \centering
    \begin{minipage}{0.5\textwidth}
    \centering
                \include{Graphics/Ind__1,0_1-1}
                \caption{}
                \label{Ind_(1,0)1-1}
    \end{minipage}~
    \begin{minipage}{0.5\textwidth}
    \centering
                \include{Graphics/Ind__1,0_1-2}
                \caption{}
                \label{Ind_(1,0)1-2}
    \end{minipage}
    \end{figure}   

    \begin{proof}
       If $X$ contains the $(rv,a)$ then we can construct a filling diagram with an equal number of 2-cells and equal curvature in the boundary such that $rv,v$ an $c$ have curvatures $0,1$ and $1$ so we can apply Claim~\ref{(0,1)1} to get a contradiction. If $X$ contains $(a,sv)$ then applying Claim~\ref{Cl-curv1} gives a contradiction. If $X$ contains $(b,c)$ then we can produce a filling diagram with an equal number of 2-cells and boundary curvature such that $rv,v$ and $c$ have curvatures $1,1$ and $0$ so we can apply Claim~\ref{(1,1)0} to get a contradiction. Therefore the 5-cycle $a,b,rv,sv,c$ shown in Figure~\ref{Ind_(1,0)1-2} has no valid triangulation so we get a contradiction. 
        
    \end{proof}

    \begin{claim}
        The 0-cells $rv,v,$ and $b$ cannot have curvatures $1,2$ and $-1$ as seen in Figure~\ref{Ind_(1,2)-1}.
        \label{Cl:(1,2)-1}
    \end{claim}

    \begin{figure}
        \centering
        \include{Graphics/Ind__1,2_-1}
                \caption{}
                \label{Ind_(1,2)-1}
    \end{figure}

    \begin{proof}
         First we see that if $X$ contains $(rb,sb)$ then by Lemma~\ref{lem:D8_diag->clique}, we have that $\langle r, s \rangle b$ is a clique, which contradicts our minimality assumption. Also $X$ cannot contain $(v,c)$ by the minimality of $P$. Therefore since $X$ contains the 5-cycle $sb,c,a,rb,v$ we have that $X$ contains $(a,sb)$. 

         Since $X$ contains the 4-cycle $rb,b,sv,rsv$ we see that $X$ contains $(rb,sv)$. Therefore $X$ contains the 4-cycle $rb,sv,sb,a$, which implies that $X$ contains $(sv,a)$, however this produces a contradiction by Claim~\ref{Cl-curv1}.
    \end{proof}

    \begin{claim}

    The $rv,v,$ and $b$ cannot have curvatures $2,1$ and $0$ as seen in Figure~\ref{Ind_(2,1)0-1}.      
    \end{claim}

    \begin{proof}
        We can produce the filling diagram $f'$ by replacing the 1-cell $(rb,v)$ in $f$ with $(rv,b)$. $f'$ satisfies our minimality assumption and the conditions of Claim~\ref{Cl:(1,2)-1}, therefore we get a contradiction.
    \end{proof}

    \begin{figure} [!tbp]
    \centering
    \begin{minipage}{0.5\textwidth}
    \centering
                \include{Graphics/Ind__2,1_0-1}
                \caption{}
                \label{Ind_(2,1)0-1}
    \end{minipage}~
    \begin{minipage}{0.5\textwidth}
    \centering
                \include{Graphics/Ind__2,1_-1-1}
                \caption{}
                \label{Ind_(2,1)-1-1}
    \end{minipage}
    \end{figure}   

    \begin{claim}
        If $rv,v$ and $b$ have curvatures $2,1$ and $-1$ as in Figure~\ref{Ind_(2,1)-1-1}, then there is a filling diagram $f'$ that satisfies our minimality assumption and such that $rv,v$ and $b$ have curvatures $1,2$ and $-2$, the curvature of $rb$ is increased by $1$, and all the other vertices of $\gamma$ have the same boundary curvature as they do in $f$.
    \end{claim}

    \begin{proof}
        To obtain $f'$ we replace the 1-cell $(rb,v)$ with $(rv,b)$.
    \end{proof}

    \begin{claim}
        If the curvatures of $rv,v$ and $a$ are $2,-1$ and $1$ as in Figure~\ref{Ind_(2,-11)1} , then $\gamma$ admits a filling diagram $f'$ that satisfies our minimality assumption such that the sum of the curvatures of $rv, v$ and $a$ is $1$, the curvature of $e$ is increased by 1, and all other 0-cells in $\gamma$ have the same boundary curvature as they do in $f$.
    \end{claim}

    \begin{figure} [!tbp]
    \centering
    \begin{minipage}{0.5\textwidth}
    \centering
                \include{Graphics/Ind__2,-1_1}
                \caption{}
                \label{Ind_(2,-11)1}
    \end{minipage}~
    \begin{minipage}{0.5\textwidth}
    \centering
                \include{Graphics/Ind__2,-1_2}
                \caption{}
                \label{Ind_(2,-1)2}
    \end{minipage}
    \end{figure}   

    \begin{proof}
        If $X$ contains $(e,c)$ then we can replace the 1-cell $(b,v)$ with $(e,c)$ in $f$ to get a filling diagram that satisfies our minimality assumption and the assumptions of Claim~\ref{Cl:1(0,2)}, giving a contradiction, so $X$ cannot contain $(e,c)$. If $X$ contains $(e,a)$ then we can produce a valid filling diagram with a lower total boundary curvature than $f$ contradicting our minimality assumption, so $X$ does not contain $(e,a)$. If $X$ contains $(a,b)$ then we can replace the 1-cell $(v,c)$ in $f$ with $(a,b)$ to get a filling diagram that satisfies the conditions of Claim~\ref{Cl:(2,0)0}, which would prove this claim, so instead assume $X$ does not contain $(a,b)$. Then since $X$ contains the 5-cycle $a,c,b,e,rv$, we see that $X$ contains $(rv,b)$ and $(rv,c)$. Then we can produce the filling diagram $f'$ which agrees with $f$ everywhere except for the changes shown in Figure~\ref{Ind_(2,-1)2} and $f'$ satisfies the claim so we are done.   
    \end{proof}

    Now we can prove Claim~\ref{cl:E}.

    \begin{proof} [Proof of Claim~\ref{cl:E}]
        All the combinations of boundary curvatures of the three vertices of each $E_i$ that sum up to more than $1$ have been show to adhere to the restrictions in Claim ~\ref{cl:E} by the previous claims. 
    \end{proof}
  \end{proof}

\section{Edge Types Between Cliques}

\label{sec:proofMain}
In this section we give a proof of Theorem~\ref{thm:main}. We introduce Theorems~\ref{thm:D8}~\ref{thm:D10}, which will guarantee that $W$ cannot act geometrically on any systolic complex $X$ containing specific subcomplexes. The formalism of Definitions~\ref{def:impl} and~\ref{def:idt} allows us to construct a way of applying Theorems~\ref{thm:D8}~\ref{thm:D10} to a more general class of systolic complexes through a procedure described in Lemma~\ref{lem:idt}.  

\begin{definition}
\label{def:impl}
	Suppose $Y$ is a simplicial complex acted on by a group $G$. We say the the edge type $E$ is \textit{elementarily implied} from a set of edge types $\mathcal{E}$ by a cycle $\gamma$ or that $\mathcal{E} \xmapsto{\gamma} E$ if there is a 4-cycle or 5-cycle $\gamma$, such that the sides of $\gamma$ are contained in $\mathcal{E}$ and the diagonals of $\gamma$ are all in the edge type $E$. We say that the edge type $E$ is \textit{chain implied} from a set of edge types $\mathcal{E}$ by a sequence of cycles $(\gamma_i)$ or that $\mathcal{E} \xRightarrow{(\gamma_i)} E$ if there exists a sequence of sets of edge types $\mathcal{E} = \mathcal{E}_0, \dots , \mathcal{E}_n $ and a sequence of edge types $E_0, \dots , E_n = E$ such that for $i \geq 1$, $\mathcal{E}_i =\mathcal{E}_{i-1} \cup \{ E_{i-1} \}$ and  $E_i$ is elementarily implied from $\mathcal{E}_i$ by $\gamma_i$. 
\end{definition}

\begin{remark}
	\label{rem:chains}
	By Remark~\ref{rm:sys -> diag}, if $X$ is a systolic complex and $\mathcal{E}$ is any subset of the edge types contained in $X$ then any edge type $E$ that is chain implied by $\mathcal{E}$ is in fact contained in $X$.
\end{remark}

\begin{notation}
	We identify the following subgroups of $W$ with their names as groups:
	\begin{align*}
		 \langle r, s \rangle & = D_8\\
		 \langle s.t \rangle & = D_{10}\\
		 \langle t,r \rangle & = D_4\\
	\end{align*}

	Let $Y$ be the 1-skeleton of the Coxeter complex of $W$. The subgroups $D_8, D_{10},$ and $D_4$ each fix unique vertices of $Y$, which together span a clique in $Y$. We label the vertices fixed by $D_8, D_{10},$ and $D_4$ as Fix$D_8$, Fix$D_{10},$ and Fix$D_4$ (see Figure~\ref{D8D10}). 
\end{notation}

\begin{theorem}
\label{thm:D8}
	Let $Y$ be the 1-skeleton of the Coxeter complex of $W$. Let $\mathcal{E}$  be the set consisting of only one edge type (Fix$D_8$, $t$Fix$D_8$). For any $n >0$ there exists an edge type $E$ chain implied by $\mathcal{E} $ between vertices of at least distance $n$ apart in $Y$.
\end{theorem}

The proof of Theorem~\ref{thm:D8} is given in Section~\ref{sec:(5,4)}.

\begin{figure} [!tbp]
	\include{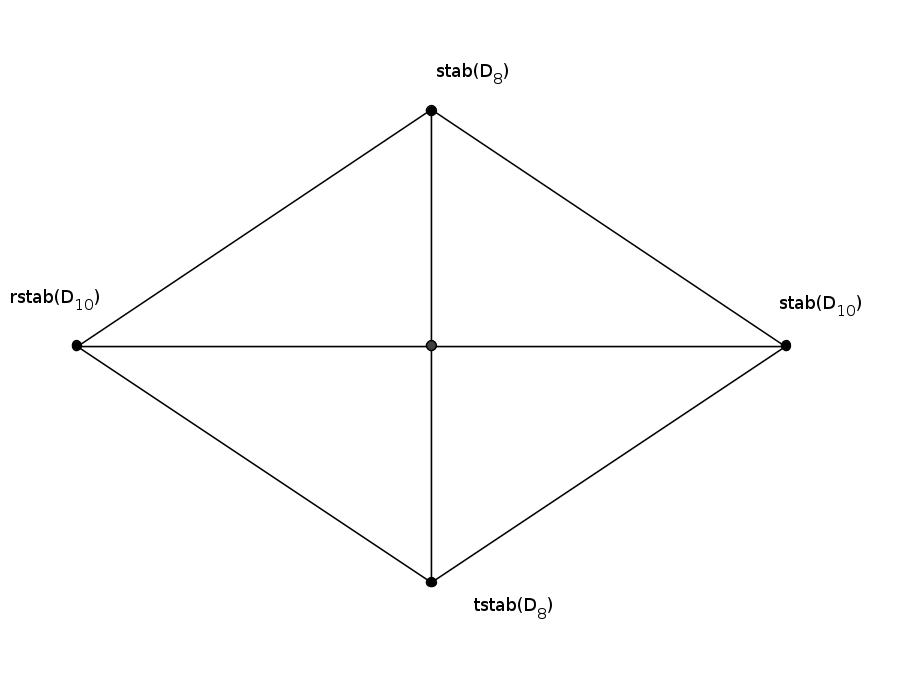}
	\caption{}
	\label{D8D10}
\end{figure}

\begin{theorem}
\label{thm:D10}
	Let $Y$ be the 1-skeleton of the Coxeter complex of $W$. Let $\mathcal{E}$ be the set consisting of only one edge type (Fix$D_{10}$, $r$Fix$D_{10}$). Then for any $n>0$, there exists an edge type $E$ chain implied by $\mathcal{E}$ between vertices of at least distance $n$ apart in $Y$.
\end{theorem}

The proof of Theorem~\ref{thm:D10} is omitted for thesis and will be prepared for the forthcoming publication. 

\begin{definition}
\label{def:cliqueEdge}
	Let $X$ be a simplicial complex and let $G$ be a group acting on $X$. Let $A$ and $B$ be cliques in $X^1$. Then we say $X$ \textit{contains the edge} $(A,B)$ or, equivalently, that $A$ and $B$ are $\textit{adjacent}$ if $A \cup B$ is a clique in $X^1$. Also, as with our notation for edges between 0-cells, we say $X$ contains the \textit{edge type} $(A,B)$ if $gA$ and $gB$ are adjacent for all $g \in G$.
\end{definition}

\begin{lemma}
\label{lem:adj}
	Suppose $X$ is a systolic complex acted on by $W$. Let $v \in X^0$ and suppose $D_8 v$, $D_{10} v$, and $D_4 v$ are cliques then they are pairwise adjacent in $X$.
\end{lemma}

The proof of the Lemma~\ref{lem:adj} will be given in Section~\ref{sec:connectingCliques}. Directly following from Theorem~\ref{Thm:fix} and Lemma~\ref{lem:adj} we get the following result. 

\begin{corollary}
	\label{cor:int}
	Suppose $X$ is a systolic complex acted on geometrically by $W$. Then there exists $v \in X^0$ such that  $D_8 v$, $D_{10} v$, and $D_4 v$ are pairwise adjacent cliques.	
\end{corollary}

\begin{proof}
	By Theorem~\ref{Thm:fix} there exists cliques $A$, $B$, and $C$ fixed respectively by $D_8$, $D_{10}$, and $D_4$ and $v \in A \cap B \cap C$. Then $D_8 v \subset A$, $D_{10} v \subset B$, and $D_4 v\subset C$ so all are cliques. Therefore by Lemma~\ref{lem:adj} $D_8 v$, $D_{10} v$, and $D_4 v$ are pairwise adjacent.	
\end{proof}

\begin{definition}
\label{def:idt}
	Suppose $X$ is a systolic complex, which is acted on geometrically by $W$. Let $Y$ be the $1$-skeleton of the Coxeter complex of $W$. By Corollary~\ref{cor:int} there exists $v \in X^0$ such that $D_8 v$, $D_{10} v$, and $D_4 v$ are pairwise adjacent. We identify vertices in $Y$ with some cliques in $X$ by identifying Fix$D_8$, Fix$D_{10}$, and Fix$D_4$ with $D_8 v$, $D_{10} v$, and $D_4 v$. We then extend the identification to all the vertices of $Y$ using the group action. This identification also gives an identification between edge types of $Y$ according to Definition~\ref{Def1} with some edge types between cliques of $X$ according to Definition~\ref{def:cliqueEdge}. We note that the identification depends on the choice of $v \in X^0$, which we call the \textit{root} of the identification.
\end{definition}

\begin{lemma}
\label{lem:idt}
	Suppose $X$ is a simplicial complex acted on geometrically by $W$. Let $Y$ be the $1$-skeleton of the Coxeter complex of $W$. Suppose $X$ contains the set of edge types $\mathcal{E}$ and $\mathcal{E}$ chain implies $E$ in $Y$. Then $X$ contains $E$.  
\end{lemma}

\begin{proof}
	We prove the Lemma for elementary implications. Since chain implications are composed of finitely many elementary implications we can just recursively apply our proof to expand our result to all chain implications.

	If $E$ is elementarily implied by $\mathcal{E}$ in $Y$, then there exists a cycle $\gamma$ in $Y$ with length $4$ or $5$ such that the sides of $\gamma$ are in $\mathcal{E}$ and the diagonals are in $E$. 

	First suppose $\gamma$ has length $4$. We label the vertices of $\gamma$ by $a,b,c,d$ and the corresponding cliques in $X$ by $A, B, C, D$. We suppose for contradiction that $X$ does not contain some edge $(v,w)$, where $v \in A$ and $w \in C$. By assumption $(a,c)$  and $(b,d)$ are in $E$ so there exists $g \in W$ such that $g(a,c) = (b,d)$. Therefore $g(v,w)$ has ends in $B$ and $D$. We see that $X$ contains the 4-cycle $v,gw,w,gv$, but by assumption both the diagonals of this $4$-cycle are not contained in $X$ so we get a contradiction.

	Now suppose $\gamma$ has length $5$. We label the vertices of $\gamma$ by $a,b,c,d,e$ and the corresponding cliques in $X$ by $A,B,C,D,E$. We suppose for contradiction that $X$ does not contain some edge $(v,w)$, where $v \in A$ and $w \in C$. Since all the diagonals of $\gamma$ are in a single edge type $E$ we have that there exists $g_1, g_2, g_3 g_4 \in W$ such that $g_1v \in B$, $g_1w \in D$, $g_2v \in C$, $g_2 w \in E$, $g_3v \in D$, $g_3 w \in A$, $g_4v \in E$, and $g_4 w \in B$. $X$ contains the $5$-cycle $g_3w,g_4w,g_2v,g_3v,g_2w$. Since $X$ does not contain diagonals $(g_3w,g_3v)$ and $(g_2w, g_2v)$ we see that $X$ contains $(g_4w, g_2w)$. Also, $X$ contains the $5$-cycle $g_3w, g_4w, g_2v, g_3v,g_4v$. Since $X$ does not contain diagonals $(g_4w,g_4v)$ and $(g_3v,g_3w)$ we see that $X$ contains $(g_2v,g_4v)$. Therefore $X$ contains the $4$-cycle $g_4v, g_2v, g_4w, g_2w$. However both diagonals of the $4$-cycle are in the edge type $(v,w)$ so we get a contradiction. 
\end{proof}

 Using the above results we can now prove Theorem \ref{thm:main}.

\begin{proof}[Proof of Theorem \ref{thm:main}]
	Assume for contradiction that $W$ is systolic, then there exists some systolic complex $X$ acted on geometrically by $W$. By Corollary~\ref{cor:int} there exist $v \in X^0$ such that  $D_8 v$, $D_{10} v$, and $D_4 v$ are pairwise adjacent cliques. 

	Now we claim that $X$ either contains $(D_8 v, t D_8 v)$ or $(D_{10} v, r D_{10} v)$. Suppose $X$ does not contain $(D_8 v, t D_8 v)$. Then there exist $a,b$ in $D_8 v$ and $t D_8 v$ respectively such that $X$ does not contain $(a,b)$. Now let $c$ and $d$ be any pair of verices from $D_{10} v$ and $r D_{10} v$, respectively. Then $X$ contains the 4-cycle $a,c,b,d$. Therefore, $X$ contains $(c,d)$ and since $c$ and $d$ were chosen arbitrarily we have that $X$ contains $(D_{10} v, r D_{10} v)$ finishing the proof of the claim.

	Let $\mathcal{E}$ be the set edge types of cardinality one containing either $($Fix$D_8, t$Fix$ D_8)$ or $($Fix$D_{10}, r $Fix$ D_{10})$ depending on whether $X$ contains $(D_8 v, t D_8 v)$ or $(D_{10} v, r D_{10} v)$. Then $\mathcal{E}$ chain implies edges of arbitrary distance in $Y$ by either Theorem~\ref{thm:D8} or Theorem~\ref{thm:D10}. Therefore, by Lemma~\ref{lem:idt}, $X$ contains infinitely many edge types so the action of $W$ cannot be cocompact, which gives a contradiction.
\end{proof}

% \begin{proof}
% 	First assume that $\gamma = v_{g_1}, \dots v_{g_n}$ is a 4-cycle. For contradiction assume that $\cup g_i K v$ is not a clique and that the nonadjacent vertices are $g_1kv$ and $g_3lv$ for $k,l \in K$. 
% \end{proof}

\section{Showing The Cliques are Pairwise Adjacent}
\label{sec:connectingCliques}

\begin{proof}[Proof of Lemma \ref{lem:adj}]

We can map the right Cayley graph of $W$ onto the subcomplex $Wv$ of $X$ in the following way:

\[
	g \rightarrow gv.
\] This map obviously preserves the left action of $W$. For the rest of the section we use labels for vertices from the right Cayley graph. If there exists some $n$-cycle in the right cayley graph we can obviously use our map to produce an $n$-cycle in $X$. We label all edge types by a single vertex and take the edge type's other end always to be $e$. For instance for $g \in W$, we label the edge type $(e,g)$ as $g$.   

We recall Remark~\ref{rem:chains} that if $X$ contains $\mathcal{E}$ and  $\mathcal{E}$ elementary implies $E$, then $E$ is contained in $X$.  We prove that $X$ contains a number of edges using this remark. We list a number of elementary implications below. The cycles used can be seen in Figure~\ref{rightCayley}. 

\begin{figure}
	\include{Graphics/rightCayley}
	\caption{}
	\label{rightCayley}
\end{figure}

\begin{align*}
\{tr, tst, rsr \} & \xmapsto{e, tr,tsr,tst} tsr\\
\end{align*}
This elementary implication indicates that if $X$ contains the edge types $(e,tr), (e,tst),$ and $(e,rsr)$ then $X$ contains the $4-$cycle $e,tr,tsr,tst$. In addtion every diagonal of the $4-$cycle is in the edge type $(e,tsr)$. Since $X$ contains $(e,tr), (e,tst),$ and $(e,rsr)$ by assumption using this implication and Remark~\ref{rem:chains} we see that $X$ contains $(e,tsr)$. The reader is advised to use this explanation to help them parse the following list of elementary implications.
\begin{align*}
\{ts,srs,r \} & \xmapsto{e,ts,trs,r} trs \\
\{tsr,srs,r \} & \xmapsto{e,tsr,trsr,r} trsr \\
\{rs,tstst,t\} & \xmapsto{e,rst,trst,t} trst \\
\{trst \} & \xmapsto{e,trst,trsrst,tsrt} trsrst \\
\{trsr \} & \xmapsto{e,trsr,rtstsr,rststr,rstr} rtstsr \\
\{trsrst, r\} & \xmapsto{e,trsrst,tsrst,r} tsrst \\
\{trsrst, srsr, t \} & \xmapsto{e,trsrst,trsrs,t} trsrs \\
\{tsrst, srs, t \} & \xmapsto{e,tsrst,tsrs,t} tsrs \\
\{rtstsr, rst, r\} & \xmapsto{e,rtstsr,rtsts, rst} rtsts \\
\{rtstsr, rsr, t\} & \xmapsto{e, rtstsr, rtststr, rsr} rtststr\\
\{rtststr, r, tstst \} & \xmapsto{e,rtststr, rtstst,r} rtstst\\
\{trs,tstst,rsr \} & \xmapsto{e,trs,rsts,rsr} rsts\\
\{trst,sts,rsr\} & \xmapsto{e,trst,rstst,rsr} rstst\\
\{srs, sts \} & \xmapsto{e,srs,strs,sts} strs \\
\{rsr, strs, tstst\} & \xmapsto{e,strs,srsts,rsr} srsts\\
\{sts, srst, tstst \} &\xmapsto{e,srst,strst,sts} strst\\
\{sts, strst, rsr \} &\xmapsto{e,strst,srstst,rsr} srstst\\
\{sts, srsts, tst\} & \xmapsto{e,sts,strsts,srsts} strsts\\
\{rsr, strsts, s\} & \xmapsto{e,strsts,srststs,rsr} srststs\\
\{srs, srsts, r \} & \xmapsto{e,srsts, rsrsts,r} rsrsts\\
\{rsrsts, tstst, s\} & \xmapsto{e,rsrstst,rstrs,s} rstrs\\
\{rstst, tsrst, srs \} & \xmapsto{e,rstst, rstrst, srs} rstrst \\
\{s, rstrst, sts \} & \xmapsto{e,rstrst,rsrstst,s} rsrstst\\
\end{align*}

We see that since $X$ contains $rsr$ that $D_{10}r$ is a $10$-cycle. Since the action of $D_{10}$ on the 10-cycle satisfies the conditions of Lemma~\ref{lem:10cycle-diag->clique} and since $X$ contains the edge $trsr$ a diagonal of the 10-cycle, we see that $r D_{10}$ is a clique. In particular $X$ contains $rstsr$. Now, we show a couple more edges are contained in $X$ using elementary implications.

\begin{align*}
\{rstsr, tst, rsrsts \} & \xmapsto{e,rsrsts, rstrsts, rstsr} rstrsts\\
\{rstrsts, s, t\} & \xmapsto{e,rstrsts, rstrstst,s} rstrstst\\	
\end{align*}

We have shown that $X$ contains the necessary edges for $D_8, D_{10},$ and $D_4$ to be pairwise adjacent, so we are done.
% \end{align*}

%  \begin{align*}
% 	 	\{ srs , st , trt , st \} & \mapsto srt \\
% 	 	\{ sr, rtr, sr, t \} &\mapsto tsr \\
% 	 	\{ s, trs, srsrs, srt \} &\mapsto stst \\
% 	 	\{t, srt, trtrt, trs \}&\mapsto trst \\
% 	 	\{strt, strt, strt, strt \}&\mapsto strsrt \\
% 	 	\{strs, strs, strs, strs, strs, strs \}&\mapsto strtrs \\
% 	 	\{s, strsrt, s strsrt \}& \mapsto trsrt \\
% 	 	\{t, trsrt, t, trsrt \}& \mapsto rsrt \\
% 	 	\{t , strsrt, t , strsrt \}& \mapsto srsrt \\
% 	 	\{t, strtrs, t , strtrs \}& \mapsto srtrs \\
% 	 	\{s, srtrs, s srtrs \}& \mapsto rtrs \\
% 	 	\{s, strtrs, s , strtrs \}& \mapsto trtrs \\
% 	 	\{strt, s , trts , trsrt \}& \mapsto strst \\
% 	 	\{strst, sr, rtrtr, rs \}& \mapsto strtr \\
% 	 	\{strst, srt, trtrtrt, trs \}& \mapsto strtrt \\
% 	 	\{rsr, rtr, rsr, rtr \}& \mapsto rstr \\
% 	 	\{rstr, trtrt, rtsr, srs \}& \mapsto rsrtr \\
% 	 	\{srsrs, srtr, rsr, rtrs \}& \mapsto rstrs\\
% 	 	\{trtrt, trsr, rtr, rsrt \}& \mapsto rtsrt\\
% 	 	\{rtsrt, rtr, trtsr, srs \}& \mapsto rsrtrt\\
% 	 	\{rtrsr, rtr, rsrtr, trt \}& \mapsto rtrstr\\
% 	 	\{rstrtr, t, rtrstr, srs \}&\mapsto rstrtrt\\
% 	 	\{r, srstr, rtrtr, rtsrs \}&\mapsto srsrtr\\
% 	 	\{srtrt, trsrt, trtrs, srsrs \}&\mapsto srstrt\\
% 	 	\{srstrt, rtr, trtsrs, r \}&\mapsto trtsrsr\\ 	
% 	 \end{align*
	 \end{proof}

\section{Finding cycles in the (5,4) Pentagonal Tiling}

\label{sec:(5,4)}

This section is concerned with proving Theorem~\ref{thm:D8}. To prove the theorem first we develop a notation to describe edge types in $Y$. Then we describe sequences of 4-cycles and 5-cycles that give us the elementary implications. Then we give an inductive argument to produce the necessary chain implication.
% The proof is divided into two parts. In the first part we show the existence of six sequences of polygons. There is one 
% sequence of pentagons, one sequence of rectangles, one sequence of squares, and three sequences of Traprilaterals. All the polygons in the sequences have diagonal that share a common edge type, so if $X$ contains any of the polygons it must also contain the polygon's diagonal's edge type. In the second part we show that polygons in the sequence can be constructed from the diagonals of the smaller polygons, thus showing that $X$ must contain all of the polygons in the six sequences. Finally we will show that since one of the sequences of polygons contains an edge of increasing length, so the action on $X$ must fail to be proper.

\begin{notation}
	\label{not}
	 We identify strings $(c_i)_{i=0}^n$ in the language $\{ L, S,  R \}$ with edge types in the following way:

	First we inductively construct the path $ P= (v_i)_{i=0}^{n+2}$ as follows

	$v_0= $Fix$D_8$

	$v_1 = t$Fix$D_8$

	For $i \geq 2$ let $v_i$ be the unique vertex in $W$Fix$D_8$ such that the 1-cell with ends $v_i$ and $v_{i+1}$ is at (clockwise) angle $\theta$ from the $1$-cell with ends $v_{i-1}$ and  $v_i$ in the link of $v$ where $\theta = \frac{\pi}{2}, \pi,$or $ \frac{3\pi}{2}$ if $c_{i-2} =L,S,$ or $R$ respectively (see Figure~\ref{pathPic}).

	We identify $(c_i)_{i=0}^n$ with the edge type $($Fix$D_8, v_{n+2})$.
\end{notation}

\begin{figure}
\include{Graphics/pathPic}
\caption{}
\label{pathPic}
	
\end{figure}

\begin{remark}
	 A string $(c_i)$ is not unique in describing an edge type. For instance we can exchange all the occurrences of $R$ and $L$ in $(c_i)$ to get a new string $({c'}_i)$ that refers to the same edge type as $(c_i)$. For example $lrlsr = rlrsl$. If we reverse $(c_i)$ then the resulting string refers to the same edge type. For example $rllss = ssllr$.
\end{remark}

\begin{remark}
	The paths described in Notation~\ref{not} all lie in the subcomplex of $Y$ with 0-cells $W$Fix$D_8$ and all the 1-cells with ends $w$Fix$D_8$ and $wt$Fix$D_8$ for $w \in W$. We can identify this subcomplex with the tiling by right-angled pentagons of the hyperbolic plane. For the rest of the section in the figures we only show this subcomplex.
\end{remark}

Now using our notation we describe 6 sequences of edge types.
\[
	(a_n)_{n=0}^\infty = \epsilon, SR, \dots , S^nRS^{n-1}, \dots
\]

\[
	(b_n)_{n=0}^\infty = R, SRL, \dots , S^nRLS^{n-1}, \dots
\]

\[
	(c_n)_{n=0}^\infty = R, SRS, \dots ,S^nRS^n, \dots
\]

\begin{figure}
\begin{minipage}{0.35\textwidth}
                 \include{Graphics/4,4-Pentagon/an}
	\caption{$a_n$}
	\label{an}
            \end{minipage}%
\begin{minipage}{0.35\textwidth}
               \include{Graphics/4,4-Pentagon/bn}
	\caption{$b_n$}
	\label{bn}
            \end{minipage}%
\begin{minipage}{0.35\textwidth}
              \include{Graphics/4,4-Pentagon/cn}
	\caption{$c_n$}
	\label{cn}
            \end{minipage}
\end{figure}

\[
	(d_n)_{n=0}^\infty = S, SSS, \dots , S^{2n +1}, \dots
\]

\[
	(e_n)_{n=0}^\infty = \epsilon, R, RLR , \dots , S^{n-2}RLRS^{n-2}, \dots 
\]

\[
	(f_n)_{n=0}^\infty = \epsilon,RL, SRLS, \dots , S^{n-1}RLS^{n-1} , \dots
\]

\begin{figure}
\begin{minipage}{0.35\textwidth}
                 \include{Graphics/4,4-Pentagon/dn}
	\caption{$d_n$}
	\label{dn}
            \end{minipage}%
\begin{minipage}{0.35\textwidth}
            \include{Graphics/4,4-Pentagon/en}
	\caption{$e_n$}
	\label{en}
            \end{minipage}%
\begin{minipage}{0.35\textwidth}
             \include{Graphics/4,4-Pentagon/fn}
	\caption{$f_n$}
	\label{fn}
            \end{minipage}
\end{figure}

% \begin{align*}
% 	(a_n)_{n=0}^\infty = [(-)], [SR], \dots , [S^nRS^{n-1}], \dots\\
% 	(b_n)_{n=0}^\infty = [R], [SRL], \dots , [S^nRLS^{n-1}], \dots\\
% 	(c_n)_{n=0}^\infty = [R], [SRS], \dots ,[S^nRS^n], \dots\\
% 	(d_n)_{n=0}^\infty = [S], [SSS], \dots , [S^{2n +1}], \dots\\
% 	(e_n)_{n=0}^\infty = [(-)], [R], [RLR] , \dots , [S^{k-2}RLRS^{k-2}], \dots \\
% 	(f_n)_{n=0}^\infty = [RL], [SRLS], \dots , [S^kRLS^k] , \dots \\
% \end{algin*}

% \begin{definition}
%  	Given a k-cycle $A = (v_i)_{i=0}^k$ contained in $X$ the sides of $A$ are edges of the form $(v_i,v_{i+1})$ where $0 \leq i \leq k$ and index addition is taken mod $k$. We say that the edge $(u,v)$ is a {\it diagonal} of  $A$ if $u$ and $v$ are elements of $A$ and $(u,v)$ is not a side of $A$. 
% \end{definition}

% \begin{definition}
% 	We say that a polygon $P$, is {\it contained} in $X$ or $X_4$ if the sides of $P$ are contained in $X$ or $X_4$ respectively.
% \end{definition}

% \begin{definition}
% 	A {\it polygon type} is an equivalence class of polygons under the action of $G$.
% \end{definition}

% \begin{definition}
% 	We say that a polygon type $\mathcal{P}$, is {\it contained} in $X$ or $X_4$ if the elements of $\mathcal{P}$ are contained in $X$ or $X_4$ respectively.
% \end{definition}
Now we describe six sequences of 4-cycles and 5-cycles that admit elementary implications.

% We let $Pent_n$ denote the following sequence of pentagons in $X_4$: let $u^0_1, \dots u^0_4$ be the vertices of a pentagon in $X_4^1$. Let $e_i$ for $1 \leq i \leq 4$ be the infinite geodesic ray starting at $u^0_{i-1}$ and going through $u^0_i$ with indices taken mod $4$. We let $u^n_i$ be the vertex on $e_i$ of distance $n+1$ from $u^0_{i-1}$.
% We define $Sq_n$ to be the pentagon with vertices $u^n_0, \dots , u^n_i$. The first few polygons in the sequence or shown in the figure.

$Pent_n$ is the sequence of 5-cycles whose first 3 iterations are shown in Figure~\ref{PentA}. The sequence gives the following sequence of elementary implications:

\[
	\{ a_n \} \xmapsto{Pent_n} b_n
\]

\begin{figure}[h!]
	\include{Graphics/4,4-Pentagon/PentA}
	\caption{$Pent_n$}
	\label{PentA}
\end{figure}

$Sq_n$ is the sequence of 4-cycles whose first 3 iterations are shown in Figure~\ref{Sq}. The  
sequence gives the following sequence of elementary implications: 

\[
	\{c_n \} \xmapsto{Sq_n} d_n
\]
\begin{figure}
	\include{Graphics/4,4-Pentagon/SquareA}
	\caption{$Sq_n$}
	\label{Sq}
\end{figure}

%$Pent_0$ is has sides of length given by the representative $[(-)]$ and is realized as one of the pentagons making up the tiling of $X_4$. $Pent_{n+1}$ is constructed by taking the sides of $Pent_{n}$ edges to either sides to get edges of type $[R^n S^n]$. The first few iterations of the sequencs are shown below 

% We let $Sq_n$ denote the following sequence of squares in $X_4:$ let $v$ be a vertex in $X_4$. Let $u^0_0, \dots , u^0_3$ be the four vertices of $X_4$ distance $1$ from $v$ in the one skeleton of $X_4$. We let $e_i$ for $0 \leq i \leq 3$ be the infinite geodesic in $X_4^1$ determined by the points $x_i$ and $v$. Let $u^n_i$ for $n\geq 1$ and $0\leq i \leq 3$ be the vertex on $e_i$ that is distance $n$ from $u^0_i$ and distance $n+1$ from $v$. $Sq_n$ is the square with vertices $u^n_0, \dots , u^n_3$.

$Rect_n$ is the sequence of 4-cycles whose first 3 iterations are shown in Figure~\ref{Rec}. The  
sequence gives the following sequence of elementary implications:

\[
 	\{ d_n, e_{n+1} \} \xmapsto{Rect_n} f_{n+1}
 \] 

 \begin{figure} 
	\include{Graphics/4,4-Pentagon/Rectangle}
	\caption{$Rec_n$}
	\label{Rec}
\end{figure}

$TrapA_n$ is the sequence of 4-cycles whose first 3 iterations are shown in Figure~\ref{TrapA}. The sequences gives the following sequence of elementary implications:

\begin{figure}
	\include{Graphics/4,4-Pentagon/QuadA}
	\caption{$TrapA_n$}
	\label{TrapA}
\end{figure}

\[
	\{ e_{n} , a_{n} , c_{n} \} \xmapsto{TrapA_n} e_{n+1}
\]

$TrapB_n$ is the sequence of 4-cycles whose first 3 iterations are given by Figure~\ref{TrapB}. The sequence gives the following sequence of elementary implications:

\begin{figure}
	\include{Graphics/4,4-Pentagon/QuadB}
	\caption{$TrapB_n$}
	\label{TrapB}
\end{figure}

\[
	\{ e_n, f_{n+1}, c_n \} \xmapsto{TrapB_n} a_{n+1}
\]

\begin{figure}
	\include{Graphics/4,4-Pentagon/QuadC}
	\caption{$TrapC_n$}
	\label{TrapC}
\end{figure}

$TrapC_n$ is the sequence of 4-cycles whose first 3 iterations are given in Figure~\ref{TrapC}. The cycles give the following sequence of elementary implications:

\[
	\{ c_n, b_{n+1}, e_{n+1} \} \xmapsto{TrapC_n} c_{n+1}
\]

\begin{lemma}
\label{lem:(5,4) len}
	For any $n \geq 1$ $\mathcal{E} = ($Fix$ D_8, t$Fix$D_8)$ chain implies $d_n$.
\end{lemma}

\begin{proof}
  To show the lemma for every $n$ we have to produce sequences $(\mathcal{E}_i)_{i=0}^m, (E_i)_{i=0}^m,$ and $(\gamma_i)_{i=0}^m$ such that $\mathcal{E}_0= \mathcal{E}$, $E_m = d_n$, for $i \geq 1$, $\mathcal{E}_i = \mathcal{E}_{i-1} \cup \{ E_{i-1} \}$ and for $i \geq 0$ $\mathcal{E}_i \xmapsto{\gamma_i} E_i$. We prove the stronger result that in addition to the previously stated requirements, we also require that 
  \[
   	\{ a_n, b_n , c_n , e_n, f_n \} \subset \mathcal{E}_m.
   \] We will proceed by induction on $n$. For the base case, $n=1$, we first note that the  edge type corresponding to the empty string, $\epsilon= ($Fix$ D_8, t$Fix$D_8) = a_1 = e_1= f_1$ is in $\mathcal{E}$ trivially. Let $\gamma_0 = Pent_1$, so $E_0 = b_1 = c_1$. Let $\gamma_1 = Sq_1$ , so $E_1 = d_1$. This gives us a chain implication of $d_1$, satisfying our extra requirement, completing the base case.

   Now we show the induction step. By the induction hypothesis we have the chain implication $\mathcal{E} \xRightarrow{(\gamma_i)_{i=0}^m} d_{n-1}$ given by sequences $(\mathcal{E}_i)_{i=0}^m, (E_i)_{i=0}^m,$ and $(\gamma_i)_{i=0}^m$ such that
   \[
   	\{ a_{n-1}, b_{n-1} , c_{n-1} , e_{n-1}, f_{n-1} \} \subset \mathcal{E}_m .
   \] We let $\gamma_{m+1} = TrapA_{n-1}$, therefore $E_{m+1} = e_n$. Let $\gamma_{m+2} = Rect_{n-1}$, therefore $E_{m+2} = f_n$. Let $\gamma_{m+3} = TrapB_{n-1}$, therefore $E_{m+3} = a_n$. Let $\gamma_{m+4} = Pent_n$, therefore $E_{m+4} = b_n$. Let $\gamma_{m+5} = TrapC_{n-1}$, therefore $E_{m+5} = c_n$. Let $\gamma_{m+6} = Sq_n$, therefore $E_{m+6} = d_n$. Therefore we have produced a chain implication $\mathcal{E} \xRightarrow{(\gamma)_{i=0}^{m+6}} d_n$ that satisfies our additional requirement, so we are done. 
\end{proof}

As a direct consequence of Lemma~\ref{lem:(5,4) len} we can prove Theorem~\ref{thm:D8}.

\begin{proof}[Proof of Theorem~\ref{thm:D8}]
	For any $n$ there exists $m$ such that $d_m$ is between vertices at least distance $n$ distance in $Y$. By Lemma~\ref{lem:(5,4) len} the set $\mathcal{E} = \{($Fix$D_8$, $t$Fix$D_8$)$\}$ chain implies $d_m$. This completes the proof.
\end{proof}

The proof of Theorem~\ref{thm:D10} uses the same  techniques, but is more complex and requires more cycles. We omit the proof, but it will be included upon publication.

\end{document}